\documentclass[12pt]{article}
\usepackage[utf8]{inputenc}
\usepackage[T1]{fontenc}
\usepackage{lmodern}
\usepackage[english]{babel}

\usepackage[a4paper,vmargin={2.5cm,2.5cm},hmargin={2.5cm,2.5cm}]{geometry}
\usepackage{mathpazo}
\usepackage[font={small,sf}, labelfont={sf,bf}, margin=1cm]{caption}
\usepackage[pdftex]{hyperref}
\usepackage[expansion]{microtype}
\usepackage[pdftex]{color,graphicx}
\usepackage{amsmath,amsfonts,amssymb,amsthm,mathrsfs}
\usepackage{stackrel}

\theoremstyle{plain}
\newtheorem{step}{Step}
\newtheorem{coro}{Corollary}[section]
\newtheorem{theo}[coro]{Theorem}
\newtheorem{prop}[coro]{Proposition}
\newtheorem{conj}[coro]{Conjecture}
\newtheorem{lemm}[coro]{Lemma}

\newtheorem*{rema}{Remark}
\newtheorem*{CMPenv}{Cluster Merging Procedure (CMP)}
\newtheorem*{SEAenv}{Stabiliser exploration algorithm}

\newtheorem{defi}[coro]{Definition}

\newtheorem{mode}{Model}

\newtheorem*{ack}{Acknowledgments}

\newcommand{\defeq}{\overset{\hbox{\tiny{def}}}{=}}

\newcommand{\ind}{\mathbf{1}}
\renewcommand{\P}{\mathbf{P}}
\newcommand{\E}{\mathbf{E}}
\newcommand{\R}{\mathbb{R}}

\newcommand{\N}{\mathbb{N}}
\newcommand{\Z}{\mathbb{Z}}

\newcommand{\kA}{\mathcal{A}}
\newcommand{\kD}{\mathcal{D}}
\newcommand{\kC}{\mathcal{C}}
\newcommand{\kG}{\mathcal{G}}

\newcommand{\kP}{\mathcal{P}}

\newcommand{\kE}{\mathcal{E}}

\newcommand{\kS}{\mathcal{S}}

\newcommand{\lin}{\left[\kern-0.15em\left[}
\newcommand{\rin} {\right]\kern-0.15em\right]}
\newcommand{\linf}{[\kern-0.15em [}
\newcommand{\rinf} {]\kern-0.15em ]}
\newcommand{\ilin}{\left]\kern-0.15em\left]}
\newcommand{\irin} {\right[\kern-0.15em\right[}

\newcommand{\diam} {\text{\textup{diam}}}

\newcommand{\cmp}{\mathscr{C}}

\hypersetup{
    pdfauthor   = {L. Ménard, A. Singh},%
    pdftitle    = {Percolation by cumulative merging and phase transition for the contact process on random graphs.}}

\begin{document}

\title{\bf Percolation by cumulative merging and\\
phase transition for the contact process\\
on random graphs.}
\author{\textsc{Laurent Ménard}\footnote{Laboratoire Modal'X, Université Paris X, France. Partially supported by grant ANR-14-CE25-0014 (ANR GRAAL). email:
\href{mailto:laurent.menard@normalesup.org}{laurent.menard@normalesup.org}}
\, and  \textsc{Arvind Singh}\footnote{Laboratoire de Mathématiques, Université Paris XI, France. email:
\href{mailto:arvind.singh@math.u-psud.fr}{arvind.singh@math.u-psud.fr}}
}
\date{}
\maketitle


\vfill

\begin{abstract}
Given a weighted graph, we introduce a partition of its vertex set such that the distance between any two clusters is bounded from below by a power of the minimum weight of both clusters. This partition is obtained by recursively merging smaller clusters and cumulating their weights. For several classical random weighted graphs, we show that there exists a phase transition regarding the existence of an infinite cluster.

The motivation for introducing this partition arises from a connection with the contact process as it roughly describes the geometry of the sets where the process survives for a long time. We give a sufficient condition on a graph to ensure that the contact process has a non trivial phase transition in terms of the existence of an infinite cluster. As an application, we prove that the contact process admits a sub-critical phase on $d$-dimensional random geometric graphs and on random Delaunay triangulations. To the best of our knowledge, these are the first examples of graphs with unbounded degrees where the critical parameter is shown to be strictly positive.
\end{abstract}

\vfill


\vfill


\newpage

\tableofcontents

\section{Introduction}

The initial motivation of this work is the study of the contact process on an infinite graph with unbounded degrees. The contact process is a classical model of interacting particle system  introduced by Harris in \cite{Harris}. It is commonly seen as a model for the spread of an infection inside a network. Roughly speaking, given a graph $G$ with vertex set $V$, the contact process on $G$ is a continuous time Markov process taking value in $\{0,1\}^V$ (sites having value $1$ at a given time are said to be \emph{infected}) and with the following dynamics:
\begin{itemize}
\item Each infected site heals at rate $1$.
\item Each healthy site becomes infected at rate $\lambda N$ where $\lambda >0$ is the infection parameter of the model and $N$ is the number of infected neighbours.
\end{itemize}
We give a rigorous definition of the contact process in Section \ref{sec:connectioncontact} and refer the reader to the books of Liggett \cite{Li1,Li2} for a comprehensive survey on interacting particle systems, including the contact process. Durett's book \cite{D} also provides a nice survey on these models in the setting of random graphs. An important feature of the model is the existence of a critical infection rate $\lambda_c$ such that the process starting from a finite number of infected sites dies out almost surely when $\lambda < \lambda_c$ but has a positive probability to survive for all times as soon as $\lambda > \lambda_c$.

\medskip

It is a general result that, on any infinite graph, there exists a super-critical phase \emph{i.e} $\lambda_c < +\infty$ (on a finite graph, the process necessarily dies out since it takes values in a finite space with zero being the unique absorbing state). This follows, for instance, from comparison with an oriented percolation process, see \cite{Li1}. On the other hand, if the graph has bounded degrees, then there also exists a non trivial sub-critical phase, \emph{i.e.} $\lambda_c > 0$. This can be seen by coupling the contact process with a continuous time branching random walk with reproduction rate $\lambda$. Thus, the phase transition is non degenerated on any vertex-transitive graph such as $\Z^d$ and regular trees. The behaviour of the contact process on those graphs has been the topic of extensive studies in the last decades and is now relatively well understood. In particular, depending on the graph, there may exist a second critical value separating a strong and weak survival phase (see \cite{P,PS,S} for such results on trees). Again, we refer the reader to \cite{D,Li1,Li2} and the references therein for details.

\medskip

Comparatively, much less in known about the behavior of the contact process on more irregular graphs. Yet, in the last years, there has been renewed interest in considering these kind of graphs as they naturally appear as limits of finite random graphs such as Erd\H{o}s-Rényi graphs, configuration models or preferential attachment graphs \cite{BBCS,C,CS,CD,MMVY,MVY}.

\medskip

However, without the boundedness assumption on the degree of $G$, the situation is much more complicated and the existence of a sub-critical phase is not guaranteed. For example,
Pemantle \cite{P} proved that, on a Galton-Watson tree with reproduction law $B$ such that, asymptotically, $\P\{B \geq x\} \geq \exp(-x^\beta)$ for some $\beta <1$, then $\lambda_c = 0$. Thus, the degree distribution of a random tree can have moments of all orders and yet the contact process on it will still survive with positive probability even for arbitrarily small infection rates. This is a very different behavior from the one observed on regular trees with similar average degree and it indicates that the survival of the contact process depends on finer geometric aspects of the underlying graph than just its growth rate. In the case of Galton-Watson trees, we expect the critical value to be positive as soon as the reproduction law has exponential moments. This still remains to be proved, even for progeny distributions with arbitrarily light tails. In fact, to the best of our knowledge, there is no (non trivial) example of a graph with unbounded degrees for which it has been shown that $\lambda_c$ is non-zero. Worst, predictions of physicists hinting for non-zero values of $\lambda$ turned out to be wrong, see for instance \cite{CD}. The main goal of this paper is to find a sufficient condition on a graph $G$ for the contact process to have a non-trivial sub-critical phase and then give examples of classical graphs satisfying this condition.

\medskip

Let us quickly explain the difficulty met when  studying the contact process on a graph with unbounded degrees. First, the comparison between the contact process and the branching random walk becomes useless since the later process always survives with positive probability. This follows from the fact that it can survive on finite star graphs with large enough degree. Therefore, we must find another way to control the influence of sites with large degree. Those sites should be seen as ``sources'' which, once infected, will generate many new infections. Some of those infections may, in turn, reach other sites with high degree. This can lead to an amplification effect preventing the process to ever die out depending on the repartition of these sources inside the graph. In this respect, it is not so hard to find conditions for the process to survive: one just has to find groups of vertices containing enough sites with very large degree to make the process super-critical. On the other hand, in order to ensure the death of the process, it is necessary to consider the global geometry of the graph. The following heuristic is meant to shed some light on this last statement and will motivate the introduction of a particular partition of the vertex set of the graph which we call \emph{cumulatively merged partition}. This partition will, ultimately, become the main object of interest of this paper.

\bigskip

\noindent\textit{Heuristic.} It is well known that the contact process on a star graph of degree $d$ ( \emph{i.e} a vertex joined to $d$ leaves) has a survival time of exponential order (say, to simplify $\exp(d)$) when the infection parameter $\lambda$ is larger than some value $\lambda_c(d) > 0$. Now, consider the contact process on an infinite graph $G$ with unbounded degrees, and fix a very small infection rate $\lambda$ so that there are only very few sites in the graph where the contact process is locally super-critical (those with degree larger than say, $d_0$).

\begin{figure}[!t]
\begin{center}
\includegraphics[width=16cm]{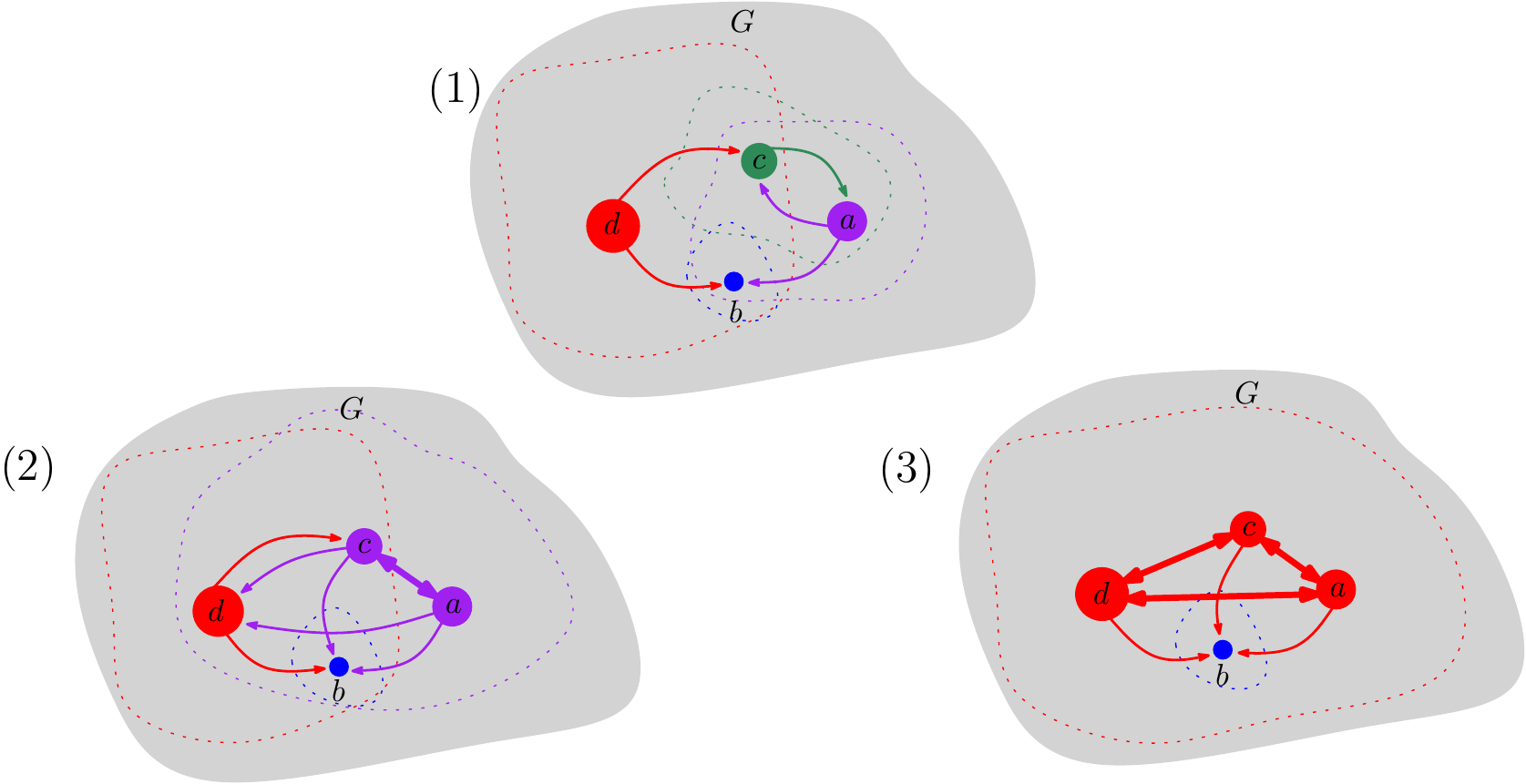}
\caption{\label{fig:heuristic} Illustration of the heuristic. In dotted lines are the maximal distances attained by infections started from each of the vertices $a$, $b$, $c$ and $d$.
Arrows represent infection fluxes and double arrows symbolize the grouping of sites, a group acting as a single site. In this example, $a$ and $c$ first merge together in (2), then the resulting cluster merges with site $d$ in (3). Note that site $b$ receives infections from the other sites but since it cannot reciprocate, it stays isolated during the merging procedure.}
\end{center}
\end{figure}

To see the influence of these vertices with anomalously large degree, imagine that we start the process with a single infected site $a$ having degree $d_a > d_0$. In addition, suppose that in a neighbourhood of $a$, every vertex has degree smaller than $d_0$.
Now run the process while forcing $a$ to stay infected for a time of order $\exp(d_a)$ after which the whole star around site $a$ recovers. By that time, roughly $\exp(d_a)$ infections will have been generated by the star around $a$. But, inside the neighbourhood of $a$, vertices have small degrees so the process is sub-critical and each infection emitted from $a$ propagates only up to a distance with finite expectation and exponential tail. This tells us that the maximal distance reached by the infections generated from $a$ should be roughly of order $d_a$.

Now imagine that within distance smaller than $d_a$ from $a$, there is some other vertex $b$ with degree $d_b > d_0$. Suppose also that $d_b$ is much smaller than the distance between $a$ and $b$. The previous heuristic applied to $b$ tells us that, in that case, the contact process started from site $b$ has little chance to ever infect $a$. Thus, infections generated by $a$ will propagate to $b$ but the converse is false. This means that, while $a$ is infected, infections regularly reach site $b$ but this flux stops when $a$ recovers, then $b$ survives for an additional time $\exp(b)$ without reinfecting $a$. So, the whole process survives for a time of order $\exp(d_a) + \exp(d_b) \approx \exp(d_a)$.

Consider now the case where there is a vertex $c$, again at distance less than $d_a$ from $a$, but this time with degree $d_c$ also larger than the distance between $a$ and $c$. In that case, infections generated by $a$ can reach $c$ and \emph{vice-versa}. Consequently, when either site $a$ or $c$ recovers, the other site has a high probability to reinfect it before its own recovery. This reinforcement effect means that, in order for the process to die out, both vertices $a$ and $c$ must recover almost simultaneously. This will happen after a time of order $\exp(d_a) \times \exp(d_c) = \exp(d_a + d_c)$. Thus, for the purpose of studying the extinction time, we can see both vertices $a$ and $c$ acting like a single vertex of degree $d_a + d_c$ (see Figure \ref{fig:heuristic} part (1) and (2)).

But now, our combined pair of vertices $(a,c)$ will send infections to a larger distance
$d_a + d_b$ and will possibly find other vertices to interact with (for example, vertex $d$ in Figure \ref{fig:heuristic}). Iterating this procedure, we recursively group vertices together, with the condition that two groups merge whenever the sum of the degrees inside each group is larger than the distance between them. Assuming that this procedure is well defined and converges, the limiting partition should satisfy the condition that, for any two equivalence classes $A$ and $B$,
\begin{equation}
\label{eq:CMPintro}
d(A,B) > \min \left\{ r(A); r(B) \right\}
\end{equation}
where $d$ is the graph distance and $r(A)$ is the sum of the degrees of the vertices of $A$.

\bigskip

In turns out that the limiting partition exists and does not depend on the order the merging procedure is performed. Its study is the purpose of Section \ref{sec:CMP} where we rigorously define it for a general weighted graph. Then, we examine some of its properties and provide a description of its internal structure. To the best of our knowledge this cumulatively merged partition (CMP) defines a new model which appears to be quite rich while still remaining amenable to analysis. We think it is interesting in its own right and might prove worthy of further investigation.

Next, in Section \ref{sec:phasetransition}, we consider various types of classical random weighted graphs for which we study the question of percolation: is there an infinite cluster in the partition? We prove that, under reasonable assumptions and similarly to classical percolation, there exists a phase transition. More precisely, we show that for i.i.d. weights, both sub-critical and super-critical phases exist on $d$-dimensional regular lattices (see Corollary \ref{coro:pc<1}, Proposition \ref{prop:CMPsubcritical} and Corollary \ref{coro:BerCMPsubcritical} for precise statements). We also prove similar results for
geometric graphs and Delaunay triangulations weighted by their degrees (Proposition \ref{prop:CMPDelaunay}). Most of the proofs in this section rely on multiscale analysis as the recursive structure of the partition  is particularly suited for this kind of arguments.

In Section \ref{sec:connectioncontact}, we return to our original question about contact process and we try to make rigorous the previous heuristic. The main result of this section, while not being completely satisfactory, gives a sufficient condition for the contact process to have a non-trivial phase transition by relating it with the existence of an infinite cluster for a particular CMP on the graph weighted by its degrees, see Theorem \ref{TheoCMPtoCP}. As an application, combining this criterion with results from Section \ref{sec:RGG and Delaunay}, we
obtain examples of graphs with unbounded degrees having a non trivial phase transition:
\begin{theo}\label{theo:CPonRGGandDelaunay} Let $G$ be either a (supercritical) random geometric graph or a Delaunay triangulation constructed from a Poisson point process with Lebesgue intensity on $\R^d$. Then, the critical infection parameter $\lambda_c$ for the contact process on $G$ is strictly positive.
\end{theo}
Finally, in Section \ref{sec:questions} we discuss some open questions and possible extensions related to the CMP and its connection with the contact process.

\section{Cumulative merging on a weighted graph}
\label{sec:CMP}

\subsection{Definition and general properties}

In the rest of this article, $G \defeq (V,E)$ will always denote a locally finite connected graph. In all cases of interest, the graph will be assumed infinite. We use the notation $d(\cdot,\cdot)$ for the usual graph distance on $G$. The ball of radius $l\geq 0$ around a vertex  $x \in V$ is denoted by $B(x_0,l) \defeq \{ x\in V, d(x_0,x) \leq l\}$. More generally, for $A,B \subset V$, we set
\begin{eqnarray*}
d(A,B) &\defeq &\inf\{ d(x,y)\; :\; (x,y)\in A\times B\},\\
B(A,l) &\defeq & \{z\in V \; : \; d(z,A) \leq l \}, \\
\diam(A) & \defeq & \sup\{d(x,y)\; : \; (x,y)\in A^2\}.
\end{eqnarray*}
The graph $G$ is equipped with a sequence of non-negative weights defined on the vertices:
$$(r(x),\; x\in V) \in [0,\infty)^V.$$
It is convenient to see $r$ as a measure on $V$ so the total weight of a set $A\subset V$ is given by
\begin{equation*}
r(A) \defeq \sum_{x\in A}r(x).
\end{equation*}
We also fix
$$
1 \leq \alpha <+\infty
$$
to which we refer as the \emph{expansion exponent}. For reasons that will become clear later, we call the quantity $r(A)^\alpha$ the \emph{influence radius} of the set $A$.
The triple $(G,r,\alpha)$ defines the parameters of our model.


We are interested in partitions of the vertex set of the graph and need to introduce some additional notation. If $\kC$ is a partition of $V$, we denote by $\sim^\kC$ the associated equivalence relation. In our setting, equivalence classes will often be referred to as \emph{clusters}. For $x\in V$, we denote by $\kC_x$ the cluster of $\kC$ containing $x$. Finally, If $\kC$ and $\kC'$ are two partitions of $V$, we will write $\kC \prec \kC'$ when $\kC'$ is coarser than $\kC$. The goal of this section is to study the partitions of $V$ which satisfy the following property:

\begin{defi} A partition $\kC$ of the vertex set $V$ of $G$ is said to be $(r,\alpha)-$\emph{admissible} if it is such that
\begin{equation}
\label{eq:proplimitCMP}
\forall C , C' \in \kC,\qquad C\neq C' \quad \Longrightarrow \quad d(C,C') >  \min\{r(C),r(C')\}^\alpha.
\end{equation}
\end{defi}

\begin{rema}
The most natural case corresponds to $\alpha = 1$ when there is no space expansion. In this case, the definition of admissibility coincides with Condition \ref{eq:CMPintro} stated in the introduction. However, in later sections, we will need results valid for general $\alpha$. Let us note that changing the expansion parameter from $1$ to $\alpha$ is not the same as merely changing every site weight to $r(x)^\alpha$. In fact, from the inequality $(\sum r(x))^\alpha \geq \sum r(x)^\alpha$, we see that $(r,\alpha)$-admissibility implies $(r^\alpha,1)$-admissibility but the converse is false in general.
\end{rema}

Let us note that the trivial partition $\{ V \}$ is always admissible for any choice of $(r,\alpha)$. Given two partitions $\kC$ and $\kC'$, we can define
\begin{equation*}
\kC \cap \kC' \defeq \{ C \cap C' : C \in \kC, C' \in \kC' \}.
\end{equation*}
More generally, given a family $(\kC^i)_{i\in I}$ of partitions, we define $\bigcap_{i\in I}\kC^i$ via the equivalence relation:
\begin{equation*}
x \overset{\bigcap_i \kC^i}{\sim} y \quad \Longleftrightarrow \quad x \overset{\kC^i}{\sim} y\;\hbox{ for all $i\in I$.}
\end{equation*}
Suppose now that every partition $\kC^i$ is admissible. Let $C \neq \widetilde{C}$ be two distinct clusters of $\bigcap_{i\in I}\kC^i$. By definition, there exist $i_0\in I$ such that $C \subset C^{i_0} \in \kC^{i_0}$ and $\widetilde{C} \subset \widetilde{C}^{i_0} \in \kC^{i_0}$ with $C^{i_0} \neq \widetilde{C}^{i_0}$. Hence
$$
d(C,\widetilde{C}) \geq d(C^{i_0},\widetilde{C}^{i_0}) > \min(r(C^{i_0}),r(\widetilde{C}^{i_0}))^\alpha \geq \min(r(C),r(\widetilde{C}))^\alpha
$$
which shows that $\bigcap_{i\in I}\kC^i$ is also admissible. Thus, being admissible is a property which is stable under intersection of partitions. This leads us to the following natural definition:

\begin{defi}\label{defCMP} We call \emph{Cumulatively Merged Partition} (CMP) of the graph $G$ with respect to $r$ and $\alpha$ the finest $(r,\alpha)$-admissible partition. It is the intersection of all $(r,\alpha)$-admissible partitions of the graph:
\[
\cmp(G,r,\alpha) \defeq \bigcap_{
\substack{
\text{$(r,\alpha)$-admissible}\\
\text{partition $\kC$ of $G$}}}
\kC.
\]
When there is no ambiguity, we will drop $G, r, \alpha$ from the notation and simply write $\cmp$.
\end{defi}

\begin{rema}
The CMP is monotone in $r$ and $\alpha$:  for $\alpha \leq \alpha'$ and for $r \leq r'$ (for the canonical partial order), we have
$$
\cmp(G,r,\alpha) \prec \cmp(G,r',\alpha').
$$
\end{rema}
The formal definition of CMP is mathematically satisfying but it is not very useful in practice. We introduce another characterization which provides an explicit algorithm for constructing $\cmp$ by repeated merging of clusters (and justifies, incidentally, the name of this partition).

\begin{defi}
\label{def:mergingoperator}
Let $x,y \in V$. The merging operator $M_{x,y}$ is the function from the set of partitions of $V$ onto itself defined by
\[
M_{x,y} (\kC) \defeq
\begin{cases}
\left( \kC \setminus \{\kC_x, \kC_y\} \right) \cup \{\kC_x \cup \kC_y\} & \text{if $\kC_x \neq \kC_y$ and $d(x,y) \leq \min(r(\kC_x),r(\kC_y))^\alpha$,} \\
\, \kC & \text{otherwise.}
\end{cases}
\]
\emph{i.e.} it outputs the same partition except for clusters $\kC_x$ and $\kC_y$ which are merged together
whenever they do not satisfy \eqref{eq:proplimitCMP}.
\end{defi}

By definition, the merging operator always returns a coarser partition than its argument:
\begin{equation}\label{monotoneMergeOp1}
\kC \prec M_{x,y}(\kC).
\end{equation}
Moreover, it is easy to verify that this operator is monotone in the following sense: for any two partitions $\kC,\kC'$ and any $x,y\in V$, we have
\begin{equation}\label{monotoneMergeOp2}
\kC \prec \kC' \quad\Longrightarrow\quad  M_{x,y}(\kC) \prec M_{x,y}(\kC').
\end{equation}
Let us also point out that $M_{x,y}$ restricted to the set of $(r,\alpha)$-admissible partitions is the identity operator. We can now state the algorithm used to construct the CMP. It formalizes the procedure described in the introduction of the paper.
\begin{CMPenv} \
\begin{itemize}
\item Fix a sequence $(x_n,y_n)_{n \in \N}$ of pairs of vertices of $V$.
\item Start from the finest partition $\kC^0 \defeq \{ \{x\}, x\in V \}$ and define by induction
$$
\kC^{n+1} \defeq M_{x_n,y_n}(\kC^n).
$$
\item For every $n$, the partition $\kC^{n+1}$ is coarser than $\kC^n$. This allows to define the limiting partition $\cmp \defeq \lim\uparrow \kC^n$ via the relation
$$
x \overset{\cmp}{\sim} y \quad \Longleftrightarrow\quad x \overset{\kC^n}{\sim} y \hbox{ for some $n$ .}
$$
\end{itemize}
\end{CMPenv}

\begin{prop} Assume that the sequence $(x_n,y_n)$ satisfies
\begin{equation}\label{eq:algosuite}
\text{$\{x_n,y_n\} = \{x,y\}$ for infinitely many $n$,}
\end{equation}
for every $x,y\in V$ with $x\neq y$.
Then the partition $\cmp$ defined by the cluster merging procedure does not depend on the choice of the sequence $(x_n,y_n)$ and coincides with the CMP of Definition \ref{defCMP}.
\end{prop}

\begin{proof}
Let $(\kC^n)$ and $(\widetilde{\kC}^n)$ be two sequences of partitions obtained by running the algorithm with
respective sequences $(x_n,y_n)$ and $(\tilde{x}_n,\tilde{y}_n)$, both satisfying property \eqref{eq:algosuite}. There exists an injection $\phi : \N \to \N$ such that $(x_n,y_n) = ( \tilde{x}_{\phi (n)}, \tilde{y}_{\phi (n)} )$ for every $n$. Using the monotonicity properties \eqref{monotoneMergeOp1} and \eqref{monotoneMergeOp2} of the merging operator, we check by induction that $\widetilde{\kC}^{\phi(n)}$ is coarser that $\kC^n$ hence $\cmp \prec \widetilde{\cmp}$ and equality follows by symmetry.

The fact that the partition $\cmp$ obtained with the algorithm is $(r,\alpha)$-admissible is straightforward since otherwise, we could find $x,y \in V$ belonging to different clusters and such that $d(x,y) \leq \min(r(\cmp_x),r(\cmp_y))^\alpha$; but then $\cmp_x$ and $\cmp_y$ would necessarily have merged at some point of the procedure thanks to \eqref{eq:algosuite}.

It remains to prove the minimality property. Let $\widehat{\kC}$ be an $(r,\alpha)$-admissible partition and let $(\widehat{\kC}^n)$ be the sequence of partitions obtained by running the cluster merging procedure with the same sequence $(x_n,y_n)$ but starting from the initial partition $\widehat{\kC}^0 = \widehat{\kC}$ instead of the finest partition. Since, the merging operator does not modify an admissible partition and since $\kC^0 \prec \widehat{\kC}^0$, the monotonicity of $M_{x,y}$ yields
$$
\kC^n \prec \widehat{\kC}^n = \widehat{\kC} \quad \hbox{for all $n$.}
$$
Hence, taking the limit $n \to \infty$, we get $\cmp \prec \widehat{\kC}$ and so $\cmp$ is indeed the finest admissible partition.
\end{proof}

We start our study of the CMP by collecting some easy properties of this partition.
\begin{samepage}
\begin{prop}\label{prop:easycmp}\
\begin{enumerate}
\item\label{easypropcmp1} Any site $x\in V$ with $r(x) < 1$ is isolated in the CMP \emph{i.e.} $\{x\} \in \cmp$ (the converse is false in general).
\item\label{easypropcmp2} For any $C\in \cmp$, we have $|C| \leq \max(r(C),1)$.
\item \label{easypropcmp3} For any $C\in \cmp$, we have
the equivalence: $r(C) = +\infty \; \Leftrightarrow \;  |C| = +\infty$.
\item\label{easypropcmp4} There is at most one infinite cluster.
\end{enumerate}
\end{prop}
\end{samepage}

\begin{proof} If a site $x$ is such that $r(x) < 1$, then $d(x,y)\geq 1 > r(x)^\alpha$ for any other site $y$, hence $\{x\}$ will never merge during the CMP. This proves \ref{easypropcmp1}. Now, if a cluster contains more than one site, then according to \ref{easypropcmp1}., the radius of each of its site is at least $1$, proving \ref{easypropcmp2}. The third statement follows from \ref{easypropcmp2} and the fact that all radii are finite by hypothesis. Finally, if there were two distinct infinite clusters, they would both have infinite influence radius so they would not satisfy \eqref{eq:proplimitCMP}.
\end{proof}

\begin{figure}[!t]
\begin{center}
\includegraphics[width=11cm]{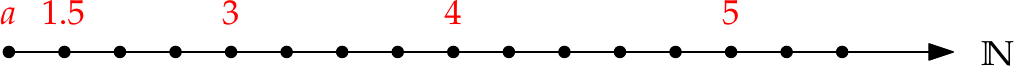}
\caption{\label{fig:changeradius} [In this example $G= \N$, $\alpha = 1$ and radii not displayed are equal to $0$]
For $a < 1$, the CMP on this weighted graph is the finest partition (it contains only isolated vertices).  For $1\leq a < 1.5$, there is a cluster with two sites and all the other sites are isolated. For $a\geq 1.5$, there is an infinite cluster composed of all the sites with non-zero radii. }
\end{center}
\end{figure}

The proposition above states that the CMP remains unchanged if all the radii $r(x) < 1$ are replaced by $0$. On the contrary, changing, even slightly, the value of a single radius with $r(x) \geq 1$ can dramatically change the structure of $\cmp$. This is illustrated in Figure \ref{fig:changeradius} and it shows that local changes can propagate to infinity.

Figure \ref{fig:changeradius} also illustrates the fact that clusters of $\cmp$ are not necessarily connected sets. In fact, they can even have asymptotic zero density.
The following proposition gives an upper bound on the diameter of the clusters
with respect to their size and will prove useful in the next sections.

\begin{prop}\label{prop_diam_weight} For any
cluster $C \in \cmp$,
\begin{equation}\label{connectedcluster}
\hbox{$B(C,r(C)^\alpha)$ is connected}
\end{equation}
and
\begin{equation}\label{ineq_weight_diam}
\diam(C) \leq \left\{
\begin{array}{ll}
\max\left(\frac{r(C)\log_2 r(C)}{2},0\right) &\hbox{ if $\alpha = 1$,}\\
\frac{r(C)^\alpha}{2^\alpha - 2} &\hbox{ if $\alpha > 1$.}
\end{array}\right.
\end{equation}
\end{prop}
\begin{proof}
Define the function
$$
f(x) \defeq \left\{
\begin{array}{ll}
\max\left(\frac{x\log_2 x}{2},0\right) &\hbox{ if $\alpha = 1$,}\\
\frac{x^\alpha}{2^\alpha - 2} &\hbox{ if $\alpha > 1$.}
\end{array}\right.
$$
We first check that $f$ satisfies the functional inequality
\begin{equation}\label{funcineq}
f(a+b) \geq f(a) + f(b) + \min(a,b)^\alpha\quad \hbox{for all $a,b\geq 1$.}
\end{equation}
To see this, fix $b > a \geq 1$ and set $z = b/a$. We can write
\begin{equation*}
f(a+b) - f(a) - f(b) - \min(a,b)^\alpha =
\left\{
\begin{array}{ll}
\frac{a}{2}\left( \log_2(1+z) + z\log(1+\frac{1}{z})-2\right)& \hbox{if $\alpha = 1$,}\\
\frac{a^\alpha}{2^\alpha - 2}\left((1+z)^\alpha - z^\alpha +1-2^\alpha\right) & \hbox{if $\alpha > 1$.}
\end{array}\right.
\end{equation*}
In both cases, the functions appearing on the right hand side of the equality are non-decreasing in $z$ and take value $0$ at $z=1$. Hence $f$ satisfies \eqref{funcineq}.

Now, \eqref{connectedcluster} and \eqref{ineq_weight_diam} are trivial when $C$ is a singleton. In particular, all the clusters of the finest partition $\{\{x\},x\in V\}$ satisfy them. Since the CMP is obtained from repeated merging operations starting from the trivial partition, it suffices to prove that \eqref{connectedcluster} and \eqref{ineq_weight_diam} are stable by the merging operator. Assume that $C_1$ and $C_2$ are distinct clusters for which \eqref{connectedcluster} and \eqref{ineq_weight_diam} hold and such that their merging is admissible:
\begin{equation}\label{eqdiamweight}
1\leq d(C_1,C_2) \leq \min(r(C_1),r(C_2))^\alpha.
\end{equation}
Then, clearly, $B(C_1 \cup C_2, r(C_1\cup C_2)^\alpha)$ is connected. Moreover, using the triangle inequality combined with \eqref{funcineq}, we find that
\begin{align*}
\diam(C_1\cup C_2) &\leq \diam(C_1) + \diam(C_2) + d(C_1,C_2) \\
&\leq f(r(C_1)) + f(r(C_2)) + \min(r(C_1),r(C_2))^\alpha\\
&\leq f(r(C_1) + r(C_2))\\
&= f(r(C_1\cup C_2)).
\end{align*}
\end{proof}

\begin{rema} The bounds of the proposition are sharp. To see this when $\alpha = 1$, consider the sequence of weights $A_n$ on $\{1,2,\ldots,n2^{n-1} \}$ defined by induction by
$$
A_1 = \boxed{11} \quad \hbox{and} \quad A_{n+1} = A_n \, \underset{\hbox{\tiny{$2^{n}-1$ zeros}}}{{\boxed{0\ldots 0}}} \, A_n.
$$
For each $n$, all the sites inside $\{1,2,\ldots,n2^{n-1} \}$ with radius $1$ merge into a single cluster $C_n$ with $r(C_n) = 2^n$ and $\diam(C_n)= \frac{n}{2}2^n$. A similar construction also works for  $\alpha > 1$.
\end{rema}

\subsection{Oriented graph structure on \texorpdfstring{$\cmp$}{C}}
\label{sec:orderoriented}

The cluster merging procedure tells us that clusters of $\cmp$ are formed by aggregation of smaller clusters. This hints that there must be some structure hidden inside the clusters (because the merging operation cannot occur in any order). In this paper, we will not be concerned with this question even though we believe it could be of independent interest. Instead, we consider the relationship between distinct clusters of $\cmp$ and show that there is a hierarchical structure which provides a natural partial order over the set of clusters.

\begin{defi} We define the relation $\mapsto$ over the elements of $\cmp$ by
$$
C \mapsto C' \quad\Longleftrightarrow \quad C\neq C' \hbox{ and }  d(C,C') \leq r(C)^\alpha
$$
for any $C,C'\in\cmp$. When this holds, we say that $C'$ descends from $C$.
\end{defi}

This relation is anti-symmetric: one cannot have $C \mapsto C'$ and $C' \mapsto C$ simultaneously because it would contradict the admissibility property of $\cmp$. We interpret $(\cmp,\cdot\mapsto\cdot)$ as an oriented graph on the set of clusters. The next proposition gathers some easy, yet important, properties concerning this graph.

\begin{prop}
\label{prop:orientedmisc}
The oriented graph $(\cmp, \cdot \mapsto \cdot)$ is such that:
\begin{enumerate}
\item The out-degree of any cluster $C$ is smaller than $|B(C,r(C)^\alpha)| - |C|$. In particular, every cluster has finite out-degree except for the eventual infinite cluster.
\item If $\cmp$ has more than one cluster, then for every cluster $C$,
$$
\hbox{$C$ has out-degree $0$}  \quad\Longleftrightarrow\quad r(C) <1
\quad\Longleftrightarrow\quad \hbox{$C$ is a singleton}.$$
\item\label{orientedmisc3} If $C \mapsto C'$, then $\lfloor  r(C)^\alpha \rfloor > \lfloor r(C')^\alpha \rfloor$ where $\lfloor x \rfloor$ denotes the integer part of $x$.
\item\label{orientedmisc4} There is no infinite oriented path $C_1 \mapsto C_2 \mapsto \ldots$ in $(\cmp, \mapsto)$ (in particular there is no oriented circuit $C_1 \mapsto C_2 \mapsto \ldots \mapsto C_1$). In addition, if $C \in \cmp$ is such that $|C| < \infty$, then every oriented path started form $C$ has length at most $\lfloor r(C)^\alpha \rfloor$.
\end{enumerate}
\end{prop}
\begin{proof}
The first statement follows from the fact that if $C \mapsto C'$, then $d(C,C') \leq r(C)^\alpha$ and therefore $C' \cap \left( B(C,r(C)^\alpha) \setminus C \right) \neq \emptyset$. The second statement is straightforward recalling that $G$ is assumed to be connected. For the third statement, notice that the graph distance $d(\cdot,\cdot)$ only takes integer values. Consequently, if $C \mapsto C'$, then
$r(C')^\alpha < d(C,C') \leq r(C)^\alpha$ which implies $\lfloor  r(C')^\alpha \rfloor < \lfloor r(C)^\alpha \rfloor$. For the last statement, if $C_1 \mapsto C_2 \mapsto \ldots$ is a chain of cluster, then $\lfloor r(C_i)\rfloor$ is a strictly decreasing sequence in $\N\cup\{+\infty\}$ so it is necessarily finite and its length is at most
$\lfloor r(C_1)\rfloor$ (or $\lfloor r(C_2)\rfloor + 1$ if $C_1$ is the infinite cluster).
\end{proof}
According to the previous proposition, the graph $(\cmp, \cdot \mapsto \cdot)$ does not contain any cycle. Thus, the relation $ \vartriangleright$ given by
$$
C \vartriangleright C' \quad \Longleftrightarrow \quad \hbox{there exists an oriented path from $C$ to $C'$}
$$
defines a partial order on $\cmp$ and Statement \ref{orientedmisc4} can be re-expressed in the form:

\begin{coro} Any totally ordered subset of $(\cmp,\vartriangleright)$ is isomorphic to one of the following ordinals: $\{1, \ldots ,n\}$, $\N$ or $\N \cup \{\infty\}$.
\end{coro}

\begin{figure}[!t]
\begin{center}
\includegraphics[width=11cm]{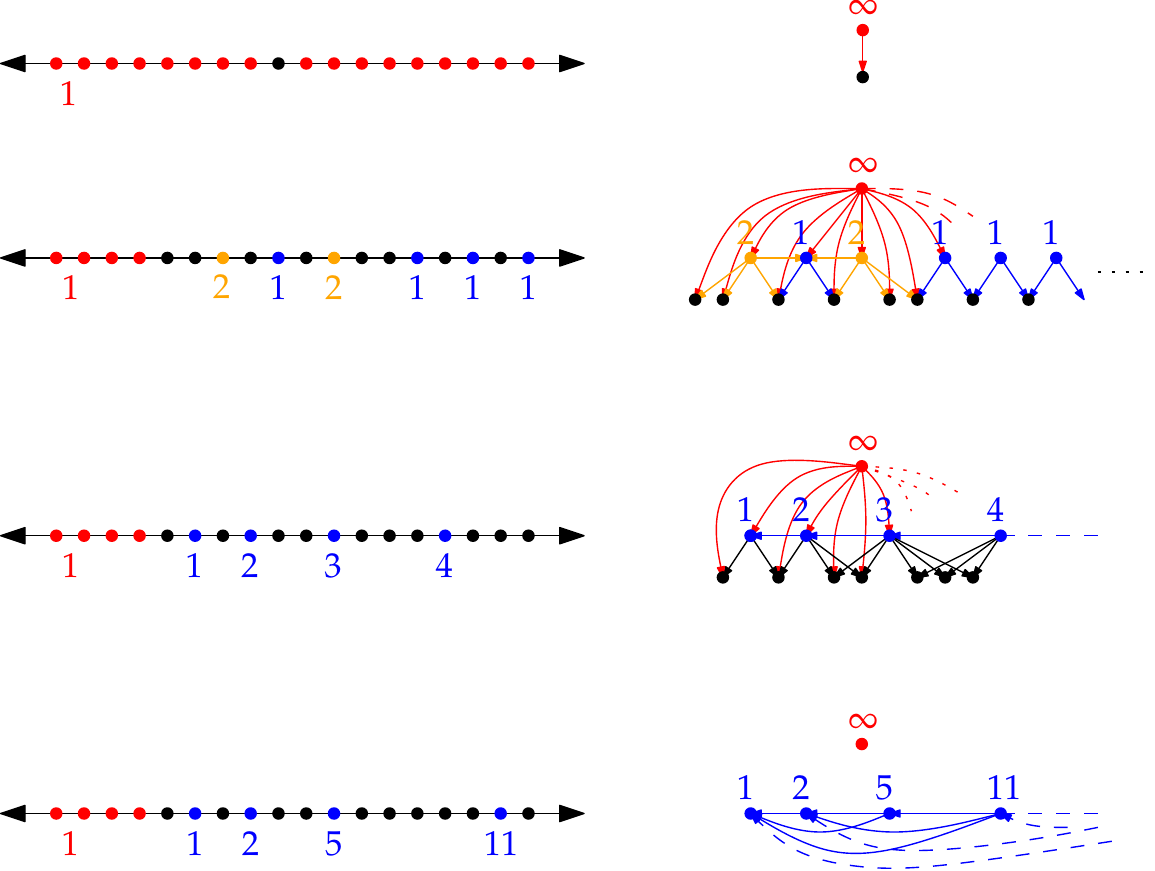}
\caption{\label{fig:oriented}Some examples of oriented graph structures on the set of clusters. In each example the underlying graph is $\Z$ and $\alpha = 1$. Red vertices all have radius $1$ and form an infinite cluster. Black vertices have radius $0$. Blue vertices have radius written below them. In every example, all clusters are a composed of a single vertex except for the infinite cluster.\\
First example: the oriented graph is finite. Second example: The graph is infinite but there is no infinite backtracking path. Third example: blue vertices form a infinite backtracking chain yet the in-degree of every cluster is finite. Fourth example: blue vertices form a infinite backtracking chain of vertices with infinite in-degrees [in this last example, arrows emanating from $\infty$ or pointing to black vertices are omitted].}
\end{center}
\end{figure}
Figure \ref{fig:oriented} gives some examples of possible oriented graph structures. It shows that, even though the out-degree of every vertex (except maybe one) is finite, in-degrees can be infinite. Let us remark that if there is one vertex with infinite in-degree, then there are infinitely many vertices  with this property and we can find an infinite sub-graph $\mathcal{H} = \{C_1,C_2,\ldots\} \subset \cmp$ such that $C_{i} \mapsto C_{j}$ if and only if $i > j$.

\subsection{Stable sets and stabilisers}

In this section, we introduce the notion of stable sets. These are subsets of $V$ such that the CMP on the inside and on the outside of the set are, in some sense, independent. Stable sets play a key role in understanding the structure of the CMP and in particular to determine whether or not there exists an infinite cluster.

Let $H$ be a subset of the vertex set of $G$. With a slight abuse of notation, we still denote by $H$ the sub-graph of $G$ induced by $H$ (\emph{i.e.} the graph with vertex set $H$ and edge set obtained by keeping only the edges of $G$ with both end vertices in $H$). We want to compare the CMP inside $H$, \emph{i.e.} $\cmp(H)$, with the trace over $H$ of the CMP on the whole graph $G$ which we denote by
$$
\cmp(G)_{|H} \defeq \{ C \cap H : \, C \in \cmp(G), \, C \cap H \neq \emptyset \}.
$$
There is an easy inclusion:
\begin{prop}
\label{prop:stablefinest}
For every $H \subset V$, we have $\cmp(H) \prec \cmp(G)_{|H}$.
\end{prop}
\begin{proof}
Let $(x_n,y_n)_{n \in \N}$ be a sequence of pairs of distinct vertices of $G$ satisfying \eqref{eq:algosuite}. We use it to construct $\cmp(G)$. We can simultaneously construct $\cmp(H)$ by considering only the indexes $n$ such that $(x_n, y_n) \in H^2$ and the result follows from the monotonicity of the merging operator.
\end{proof}
Without additional assumptions, $\cmp(G)_{|H}$ can be strictly coarser than $\cmp(H)$ since clusters growing outside of $H$ can merge with clusters growing inside of $H$ which, in turn, can yield additional merging inside $H$.
\begin{defi}
We say that a subset $H \subset V$ is \emph{stable} (for the CMP in $G$) if
\begin{equation}
\label{eq:stable}
\forall C \in \cmp(H), \quad B(C,r(C)^\alpha) \subset H.
\end{equation}
\end{defi}

\begin{rema}\
\begin{itemize}
\item Being stable is a local property: we only need to look at the weights inside $H$ to compute $\cmp(H)$ and check if it satisfies \eqref{eq:stable}. Thus, it does not depend on the  value of the weights on $G\setminus H$.
\item Since the CMP is defined as the finest $(r,\alpha)$-admissible partition, in order to show that a set $H$ is stable, it suffices to find any $(r,\alpha)$-admissible partition of $H$ satisfying \eqref{eq:stable}.
\end{itemize}
\end{rema}

\begin{figure}[!t]
\begin{center}
\includegraphics[width=12cm]{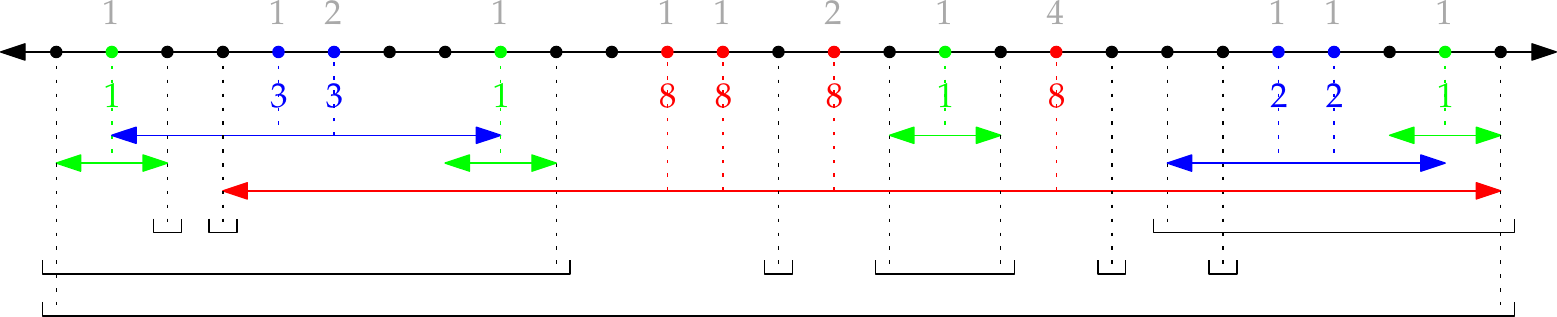}
\caption{\label{fig:stable}Examples of stable sets. Here $G=\Z$ and $\alpha = 1$. The initial weights are displayed in grey above the line (black vertices are given weight $0$). Clusters are each given a color and their zone of influence is materialized by an arrow. Some (but not all) stable sets are materialized by brackets (all the sites inside a bracket form a stable set). Note that the whole region displayed forms a stable set: the output of the CMP inside this region does not depend on the weights outside this region.}
\end{center}
\end{figure}

The following proposition highlights the importance of stable sets as its shows that they are the sets for which the CMP can be split into two separate partitions.

\begin{prop}
\label{prop:locality}
Let $H \subset V$ be stable set. Then
$$
\cmp(H) = \cmp(G)_{|H} \quad and \quad \cmp(G\setminus H) = \cmp(G)_{|G\setminus H}.
$$
Moreover, we have the decomposition
$$
\cmp(G) = \cmp(H) \sqcup \cmp(G\setminus H).
$$
\end{prop}
We point out that, even though $H$ and $G\setminus H$ seem to play a symmetric role in the proposition, $G\setminus H$ is not necessarily stable when $H$ is. A direct consequence is the following description of a stable set in terms of the CMP on $G$.
\begin{coro}\label{corostable} $H\subset V$ is stable if and only if
$$
\bigcup_{x\in H}B(x,r(\cmp_x(G))^\alpha) = H.
$$
\end{coro}

\begin{proof}[Proof of Proposition \ref{prop:locality}] Fix a sequence $(x_n,y_n)$ of pairs of vertices of $G$ satisfying \eqref{eq:algosuite}. We start from the finest partition and simultaneously build three sequences of partitions $(\kC^n(G))$ , $(\kC^n(H))$ and $(\kC^n(G\setminus H))$ using the merging procedure:
\begin{align}
\kC^{n+1}(G) & \defeq  M_{x_n,y_n} (\kC^n(G)); \\
\kC^{n+1}(H) & \defeq
\begin{cases}
M_{x_n,y_n} (\kC^n(H)) &\text{if $x_n$ and $y_n\in H$,}\\
\kC^n(H) &\text{otherwise;}
\end{cases} \\
\kC^{n+1}(G\setminus H) & \defeq
\begin{cases}
M_{x_n,y_n} (\kC^n(G\setminus H)) &\text{if $x_n$ and $y_n\in G\setminus H$,}\\
\kC^n(G \setminus H)&\text{otherwise.}
\end{cases}
\end{align}
These sequences converge respectively towards $\cmp(G)$, $\cmp(H)$ and $\cmp(G\setminus H)$. Therefore, we just need to show that for each $n$,
\begin{equation}\label{inducstable1}
\kC^{n}_z(G) = \kC^{n}_z(H) \quad\hbox{for every $z\in H$,}
\end{equation}
and that similarly,
\begin{equation}\label{inducstable2}
\kC^{n}_z(G) = \kC^{n}_z(G\setminus H) \quad\hbox{for every $z\in G\setminus H$.}
\end{equation}
Of course, we prove this by induction on $n$. Clearly, \eqref{inducstable1} and \eqref{inducstable2} hold for $n=0$ because we start from the finest partition. Assume that these equalities hold for $n$. If $x_n$ and $y_n$ are both in $H$ or both in $G\setminus H$,  then clearly the recurrence hypothesis still holds for $n+1$. It remains to check that, if $x_n\in H$ and $y_n \in G\setminus H$ (the other case is symmetric), then no merging occurs in $\kC^{n+1}(G)$. To see this, we compute
\begin{align*}
d(\kC^n_{x_n}(G) , \kC^n_{y_n}(G) ) & \geq d (\kC^n_{x_n}(G)  , G\setminus H )\\
& = d (\kC^n_{x_n}(H)  , G\setminus H ) && \text{[rec. hypothesis]} \\
& \geq d (\cmp_{x_n}(H), G\setminus H) && \text{[$\kC^n_{x_n}(H) \prec \cmp_{x_n}(H)$]} \\
& > r(\cmp_{x_n}(H))^\alpha && \text{[$H$ is stable]} \\
& \geq  r(\kC^{n}_{x_n}(H))^\alpha && \text{[$\kC^n_{x_n}(H) \prec \cmp_{x_n}(H)$]} \\
& = r(\kC^{n}_{x_n}(G))^\alpha. && \text{[rec. hypothesis]}
\end{align*}
Thus $d(x_n,y_n) > \min(r(\kC^n_{x_n}(G)) , r(\kC^n_{y_n}(G)))^\alpha$ which tells us that the clusters do not merge.
\end{proof}
\begin{prop}
\label{prop:unionstable} Let $(H_i, i\in I)$ be a family of stable subsets of $V$. Then
\[
\hat H = \bigcap_{i \in I} H_i \quad\hbox{ and }\quad \check H = \bigcup_{i \in I} H_i
\]
are also stable.
\end{prop}
\begin{proof}
Because each $H_i$ is a disjoint union of clusters of $\cmp(G)$, so are $\hat H$ and $\check H$. Fix $C \in \cmp(G)$. If $C \subset \check H$, there exists $i_0$ such that $C \subset H_{i_0}$ therefore $B(C, r(C)^\alpha) \subset H_{i_0} \subset \check H$ hence $\check H$ is stable. Similarly, if $C \subset \hat H$, then $B(C, r(C)^\alpha) \subset H_i$ for every $i\in I$ thus $B(C,r(C)^\alpha) \subset \hat H$ and $\hat H$ is also stable.
\end{proof}

 The following proposition provides a method for constructing large stable sets from smaller ones by ``dilution'': if a set is surrounded by stable subsets that are large enough, then the union of these sets is again stable. We will use this idea extensively in the next section to prove the existence of a sub-critical regime for the CMP on random graphs.
\begin{prop}\label{prop:etouffement} Let $W \subset \widetilde{W} \subset V$. Assume that
$$
\widetilde{W} \setminus W \hbox{ is a stable set}
$$
and that
$$
B(W,r(W)^\alpha) \subset \widetilde{W}.
$$
Then $\widetilde{W}$ is a stable set.
\end{prop}
\begin{proof}
We simply observe that any cluster inside $W$ has an influence radius bounded above by $r(W)^\alpha$ and therefore cannot reach outside of $\widetilde{W}$. Since the subset $\widetilde{W} \setminus W$ is stable, its clusters cannot merge with clusters inside $W$ and we find that
\begin{eqnarray*}
\bigcup_{x\in \widetilde{W}}B(x,r(\cmp_x(\widetilde{W}))^\alpha) &=&
\bigcup_{x\in \widetilde{W}\setminus W}B(x,r(\cmp_x(\widetilde{W}\setminus W))^\alpha) \cup \bigcup_{x\in W}B(x,r(\cmp_x(W))^\alpha)\\
&\subset& (\widetilde{W}\setminus W)\cup B(W,r(W)^\alpha)\\
&\subset& \widetilde{W}
\end{eqnarray*}
which shows that $\widetilde{W}$ is stable.
\end{proof}
The whole vertex set $V$ is always stable. Since stable sets are stable under intersection, it is natural to consider the smallest stable set containing a given subset.
\begin{defi} For $W \subset V$, we call \emph{stabiliser} of $W$ and denote by $\kS_W$ the smallest stable set containing $W$. It is the intersection of all stable sets containing $W$:
$$
\kS_W \defeq \bigcap_{
\substack{
\text{$H$ stable}\\
W \subset H}
} H.
$$
\end{defi}
When considering stabilisers of single vertices, we use the notation $\kS_x$ instead of $\kS_{\{x \} }$. Figure \ref{fig:stabilisators} shows some examples of stabilisers. We start by collecting some basic properties of these sets.
\begin{prop}
\label{prop:stablemisc}
Stabilisers have the following properties:
\begin{enumerate}
\item $\kS_x = \kS_{\cmp_x}$ for any $x\in V$.
\item $\kS_W \supset \cup_{x \in W} \cmp_x$.
\item $\kS_{\kS_W} = \kS_W$ and if $W_1 \subset W_2$ then $\kS_{W_1} \subset \kS_{W_2}$.
\item $\kS_{\cup_{i \in I} W_i} = \cup_{i \in I} \kS_{W_i}$ and $\kS_{\cap_{i \in I} W_i} \subset \cap_{i \in I} \kS_{W_i}$.
\item For $x \in V$, if $\kS_x = \cmp_x$, then either $r(x) < 1$ and $\kS_x = \{x\}$, or $\kS_x = V$.
\end{enumerate}
\end{prop}
\begin{proof} 1. and 2. follow from the fact that a stable set is a disjoint union of clusters of the CMP. Statement 3. is trivial and 4. follows from Proposition \ref{prop:unionstable} and the minimality of stabilisers. The last statement
is a consequence of the connectedness of $G$.
\end{proof}

\begin{prop}\label{prop:expballsstab} Let $W \in V$. Define the growing sequence of subsets $(S^n)$ of $V$ by
$$
S^0 \defeq W \quad\hbox{and}\quad
S^{n+1} \defeq \bigcup_{x\in S^n}B(x,r(\cmp_x)^\alpha).
$$
Then, we have
$$
\kS_W = \lim_n\uparrow S^n.
$$
As a consequence, if $C\in \cmp$, then $\kS_{C}$ is a connected set (although $C$ itself need not be connected).
\end{prop}
\begin{proof}
Let $x\in \lim_n\uparrow S^n$. Then, $x\in S^{n_0}$ for $n_0$ large enough hence $B(x,r(\cmp_x)^\alpha)\in S^{n_0+1} \subset \lim_n\uparrow S^n$. This means that $\lim_n\uparrow S^n$ is stable. On the other hand, by a trivial induction argument, Corollary \ref{corostable} gives that $S^n \subset \kS_W$ for all $n$. Consequently $\lim_n\uparrow S^n \subset \kS_W$ and equality follows by minimality of stabilisers. The fact that $\kS_C$ is connected for $C\in \cmp$ follows from Proposition \ref{prop_diam_weight} which ensures that $S^1$ is connected. Then $S^n$ remains connected for all $n$ since we grow these sets by adding adjacent connected sets.
\end{proof}

\begin{figure}[!t]
\begin{center}
\includegraphics[width=14cm]{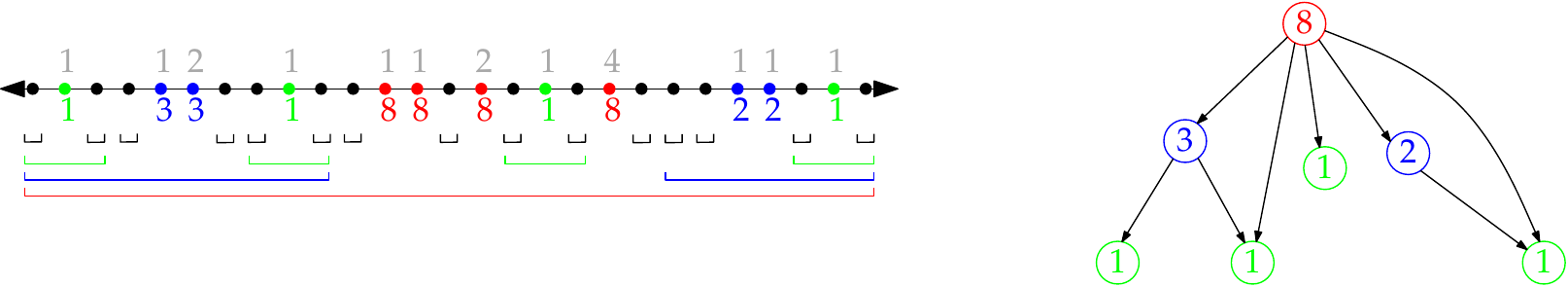}
\caption{\label{fig:stabilisators}
Examples of stabilisers. The weighted graph is the same as in Figure \ref{fig:stable}. On the left, the stabilisers of each clusters inside the region are materialized by brackets. On the right, the corresponding oriented graph $(\cmp,\cdot\mapsto\cdot)$. Vertices and arrows pointing to black clusters of null radii are not represented. Notice that the leftmost cluster of size $1$ is not a direct descendent of the cluster of size $8$.}
\end{center}
\end{figure}

A direct consequence of the previous proposition, stated in Corollary \ref{coro:stabdescend}, is that there is a nice interpretation of stabilisers of clusters in terms of the partial order $\vartriangleright$ defined in Section \ref{sec:orderoriented}. It shows in particular that these stabilisers ``pile over each other'' as illustrated in Figure \ref{fig:stabilisators}.
\begin{coro}\label{coro:stabdescend}
For any $C\in \cmp$, we have
$$
\kS_C = C\sqcup \left( \bigsqcup_{C \vartriangleright C'}C' \right)
$$
\emph{i.e.} the stabiliser of a cluster $C$ is exactly the subset composed of $C$ together with all its descendants in the oriented graph $(\cmp,\cdot \mapsto \cdot)$.
As a consequence, for any $C , C' \in \cmp$ we are in one of the three cases:
$$
\kS_C \subset \kS_{C'},\qquad \kS_C \supset \kS_{C'},\qquad \kS_C \cap \kS_{C'} = \emptyset.
$$
Moreover, we have the equivalence
$$
\kS_C = \kS_{C'} \quad \Longleftrightarrow\quad C=C'.
$$
\end{coro}

We can now state the main result of this section which will be instrumental in studying the existence of an infinite cluster in the CMP. The proof is straightforward but its usefulness promotes it to the rank of theorem.
\begin{theo}\label{theo:equivInfiniteCluster}[\textbf{\textup{Criterion for the existence of an infinite cluster}}]
Suppose that $G$ is an infinite graph. For every $x \in V$, the following statements are equivalent:
\begin{enumerate}
\item $|\cmp_x| = \infty$;
\item $|\kS_x| = \infty$;
\item $\kS_x = V$.
\end{enumerate}
Equivalently, the partition $\cmp$ has no infinite cluster if and only if there exists an increasing sequence $(S_n)$ of stable subsets such that $\lim\uparrow S_n = V$.
\end{theo}
\begin{proof}
Since $\cmp_x \subset \kS_x$, if $\kS_x$ is finite then so is $\cmp_x$. Reciprocally, suppose that $\cmp_x$ is finite, then $\{C\in\cmp\; :\;  \cmp_x\vartriangleright C\}$ is a finite set (of finite clusters) according to Proposition \ref{prop:orientedmisc}. Consequently, Corollary \ref{coro:stabdescend} implies that $\kS_x$ being the union of $\cmp_x$ and all these sets is also finite. Moreover, in this case we have $\kS_x \neq V$ since the graph is infinite. Finally, if $\cmp_x$ is infinite, then Item \ref{easypropcmp3} of Proposition \ref{prop:easycmp} asserts that its influence radius is infinite hence $\kS_x \supset B(\cmp_x,r(\cmp_x)^\alpha) = V$.
\end{proof}

\begin{rema} As we already noticed, changing the weight of a single vertex can have a macroscopic effect on the CMP. For example it can create or destroy the infinite cluster. Theorem \ref{theo:equivInfiniteCluster} tells us that, on the other hand, there is still some locality in the CMP: clusters cannot "grow" from infinity. This essential feature of the model comes from the fact that we take the minimum of both radius in the definition of admissible partitions. Therefore, for two clusters to merge, both must reach to the other one. Most of  the properties of the CMP described in this section would fail if, instead partitions satisfying \eqref{eq:proplimitCMP}, we were to consider a definition of admissibility where the minimum of the radii of the clusters is replaced by either the maximum or the sum of the radii.
\end{rema}

We now give an algorithm that explores the graph starting from a given vertex $x_0$ and reveals $\kS_{x_0}$ together with $\cmp_{|\kS_{x_0}}$ while never looking at the weight of any vertex outside of $\kS_{x_0}$. In particular, the algorithm stops after a finite number of steps if and only if the cluster containing $x_0$ is finite. The fact that this algorithm works as intended is again a direct consequence of Propositions \ref{prop:locality} and \ref{prop:expballsstab}.

\begin{SEAenv}\
\begin{enumerate}
\item Start from the set of vertices $H^0 \defeq \{x_0\}$ and the partition $\kC^0 \defeq \{ \{x_0\} \}$.
\item At the $n$-th iteration of the algorithm, we have a finite set of vertices $H^n$ and a partition of this set $\kC^n$. To go to the next step, we define
$$
H^{n+1} \defeq \bigcup_{z \in H^n} B(z,r(\kC^n_z)^\alpha) \quad \hbox{ and } \kC^{n+1} = \cmp(H^{n+1}).
$$
(in practical implementations, the partition $\kC^{n+1}$ is obtained by running the cluster merging procedure starting from the partition $\kC^n\cup\{ \{z\},z\in H^{n+1}\setminus H^n \}$ instead of the finest partition on $H^{n+1}$ so that merges from previous iterations are not repeated at each step).
\item If $H^{n+1} = H^n$, then the algorithm stops and outputs $\kS_{x_0} = H^n$ and $\cmp_{|\kS_{x_0}} = C^{n}$. Otherwise, we iterate to step $2$.
\end{enumerate}
\end{SEAenv}

\bigskip

\begin{rema}
Let us conclude our study of the general properties of the CMP by pointing out the fact that everything we established so far remains valid if we replace the expansion exponent by a general function \emph{i.e.} if we consider partitions satisfying
\[
d(A,B) > \ell \left( \min \left\{ r(A), r(B) \right\} \right)
\]
for any pair of clusters $A$ and $B$, where $\ell$ is a non-decreasing function
going to $0$ at $0$ and to $+ \infty$ at $+\infty$. The only notable difference is in Proposition \ref{prop_diam_weight} where the upper bound for the diameter of a cluster is now given in term of a function $f$ satisfying the functional equation \eqref{funcineq} with $\ell$ in place of $\alpha$.
\end{rema}


\section{Phase transitions for cumulative merging on random weighted graphs}
\label{sec:phasetransition}

In this section, we consider the CMP on several random weighted graphs. We investigate whether or not the partition $\cmp$ contains an infinite cluster. At first look, one might fear that this will always be the case due to the amplification phenomenon resulting from the additive nature of cluster merging. Or, on the contrary, the cumulative effect might be quite weak and percolation by cumulative merging could be very similar to classical site percolation. It turns out that both worries are unfounded and that, for a wide variety of random weighted graphs, there is a non-trivial phase transition differing from that of classical site percolation.

In this paper, we will consider the following three general families of graphs:

\begin{mode}[\textbf{Bernoulli CMP}]
\label{BerCMP}
The underlying graph $G$ is a deterministic infinite graph (e.g. $\Z^d$, a tree \ldots) and the weights $(r(x), x \in V)$ are independent identically distributed Bernoulli random variables with parameter $p \in [0,1]$. We denote by $\P_{p} ^G$ the law of $(G,r)$.
\end{mode}

\begin{mode}[\textbf{Continuum CMP}]
\label{BoolCMP}
The underlying graph $G$ is a deterministic infinite graph and the weights $(r(x), x \in V)$ are independent identically distributed random variables with law $\lambda Z$, where $\lambda\geq 0$ and $Z$ is a random variable taking value in $[0,\infty)$. We denote by $\P_{\lambda} ^{G,Z}$ the law of $(G,r)$.
\end{mode}

\begin{mode}[\textbf{Degree-weighted CMP}]
\label{DegCMP}
The underlying graph $G$ is a random infinite graph (e.g. a Galton-Watson tree, a random Delaunay triangulation, a random planar map, \ldots) and the weights are defined by $r(x) \defeq \deg(x)\ind_{\{\deg(x) \geq \Delta\} }$, with $\Delta \geq 0$. We denote by $\P_{\Delta}^G$ the law of $(G,r)$.
\end{mode}

The first two models are the counterparts in the context of cumulative merging of classical (site) percolation and boolean models. The third model may seem artificial at first. However, as we already explained in the introduction, it appears naturally in the connection between cumulative merging and the contact process. We will investigate this relationship in Section \ref{sec:connectioncontact}.

For all these models, we will simply write $\P$ for the law of the weighted graph when the indices are clear from the context. Each model has a free parameter ($p$ for Bernoulli, $\lambda$ for continuum, and $\Delta$ for degree-weighted CMP) so we ask, the expansion exponent $\alpha$ being fixed, whether of not $\cmp$ contains an infinite cluster depending on the value of this parameter. By monotonicity of the CMP with respect to $\alpha$ and the weight sequence $r$, the probability of having an infinite cluster is monotone in both $\alpha$ and the free parameter of the model.
\begin{defi}
For Bernoulli CMP we define
\[
p_c (\alpha) \defeq \inf \left\{ p \in [0 ; 1] \, : \,  \P_p^G \left\{ \cmp(G,r,\alpha) \, \hbox{has an infinite cluster}\right\} > 0 \right\} \in [0 ; 1].
\]
Similarly, for continuous CMP we define
\[
\lambda_c (\alpha) \defeq \inf \left\{ \lambda \geq 0 \, : \,  \P_\lambda^{G,Z} \left\{  \cmp(G,r,\alpha) \, \hbox{has an infinite cluster}\right\} > 0 \right\} \in [0 ; +\infty],
\]
and for degree biased CMP, we set
\[
\Delta_c (\alpha) \defeq \sup\left\{ \Delta \geq 1 \, : \,  \P_{\Delta}^{G} \left\{  \cmp(G,r,\alpha) \, \hbox{has an infinite cluster}\right\} > 0 \right\} \in \lin 1  ; +\infty\rin.
\]
\end{defi}
Under fairly general assumptions on $G$, it is easy to check that the existence of an infinite cluster is an event of either null or full probability.
\begin{prop}
\label{prop:0-1law}
Suppose that $G$ is a vertex transitive graph, then for Bernoulli CMP (model 1) or continuum CMP (model 2), we have
\[
\P \left\{ \cmp \hbox{ has an infinite cluster}\right\} \in \{ 0,1\}.
\]
\end{prop}
\begin{proof}
Similarly to i.i.d. percolation on transitive graphs, the result follows from ergodicity since the family of weights is invariant by translations and so is the (measurable) event of having an infinite cluster.
\end{proof}
In the case of Model $3$, one needs, of course, to make some assumptions on the random graph $G$ in order to get a 0-1 law. However, for a large class of graphs, the existence of an infinite cluster is still a trivial event thank again to general ergodicity properties. This is in particular the case for the random geometric graphs and the Delaunay triangulations considered in Section \ref{sec:RGG and Delaunay} as well as for Galton-Watson trees or unimodular random graphs.

\subsection{Phase transition on \texorpdfstring{$\Z$}{Z} for Bernoulli CMP}

By trivial coupling, it is clear that for any graph the critical parameter $p_c$ for Bernoulli CMP is smaller or equal to the critical parameter $p_{\scriptstyle{site}}$ for classical i.i.d. site percolation. We now prove that these parameter differ in general. The following
result shows that, even in dimension $1$, there exists an infinite cluster in $\cmp$ for Bernoulli CMP when $p$ is close enough to $1$. This contrasts with the case of site percolation where $p_{\scriptstyle{site}}(\Z) = 1$.

\begin{prop}\label{prop:multiscaleZ}
Consider Bernoulli CMP on $G = \N$. For any $\alpha\geq 1$, we have $p_c(\alpha) < 1$.
\end{prop}

Any infinite connected graph contains $\N$ as a sub-graph. Thus, by  coupling, the proposition above shows that there always exists a super-critical phase for Bernoulli CMP on any infinite graph. Let us also remark that, if $Z$ is a non-negative random variable which is not identically zero, then, for any $0\leq p<1$ there exists, $0\leq \lambda(p) < \infty$ such that $\lambda(p) Z$ stochastically dominates a Bernoulli random variable with parameter $p$. Putting all this together, we get

\begin{coro}
\label{coro:pc<1}
Assume that $G$ is an infinite connected graph and let $\alpha \geq 1$.
\begin{itemize}
\item $p_c(\alpha) < 1$ for Bernoulli CMP.
\item $\lambda_c(\alpha) < \infty$ for Continuum CMP as soon as $Z$ is not identically $0$.
\end{itemize}
\end{coro}
\begin{proof}[Proof of proposition \ref{prop:multiscaleZ}]
Since the CMP is monotone with respect to $\alpha$, we only need to prove the result for $\alpha = 1$ (which we assume from now on). The recursive structure of the CMP is particularly suited for using renormalization arguments so it is not surprising that most of our proofs use a multiscale analysis. In the present case, we consider events of the form
\[
\kE (n,\gamma) \defeq \left\{
\text{$\cmp(\lin0;n\irin)$ contains a cluster with a least $\gamma n$ elements}
\right\}.
\]
Note that we consider here the CMP on the sub-graph $\lin 0;n\irin$ which does not necessarily coincide with the restriction of $\cmp(\N)$ to the interval $\lin 0;n \irin$. We call these events \emph{good} because of the following inequality which holds whenever $\frac{3}{4} \leq \gamma \leq 1$:
\[
\P\left\{\kE(2n, \gamma) \right\} \geq 1 - \left( 1 - \P\left\{\kE(n, \gamma) \right\} \right)^2.
\]
To see that, apply the cluster merging procedure on $\lin 0;n \irin$ and on $\lin n;2n \irin$ separately. If each of these intervals contains a big cluster of size greater than $\gamma n$, then the rightmost vertex on the big cluster of the first interval is a distance smaller that $\frac{n}{2}\leq \gamma n$ of the leftmost vertex of the big cluster of the second interval. Therefore, these two clusters will merge together when performing the CMP on $\lin 0;2n \irin$ and give birth to a cluster of size at least $2\gamma n$.

Define $u_n \defeq 20 \cdot 2^n$ for $n \geq 0$. We also define the sequence $(\gamma_n)_{n\geq 0}$ by $\gamma_0 \defeq 9/10$ and $\gamma_{n+1} \defeq \gamma_n \left( 1 - \frac{2}{u_{n+1}} \right)$. We can write
\[
\gamma_n = \frac{9}{10} + \sum_{k=0}^{n - 1} (\gamma_{k+1} - \gamma_k) \geq \frac{9}{10} - \sum_{k=0}^{n - 1} \frac{2}{u_{k+1}} = \frac{9}{10} - \frac{2}{20} \sum_{k=1}^{n} \frac{1}{2^k} \geq \frac{9}{10} - \frac{1}{10} = \frac{8}{10}.
\]
so that  $3/4 < \gamma_n < 1$ for every $n$. Now, fix $\varepsilon > 0$, by continuity, it is possible to chose $p \in (0,1)$ close enough to $1$ such that $\P \left\{ \kE(20,9/10) \right\} > 1 - \varepsilon$. In the following lines, we will prove by induction that, for $n \geq 0$,
\[
\P \left\{ \kE(v_n,\gamma_n) \right\} > 1 - \varepsilon / 4^n
\]
where $v_n \defeq u_0 \cdot u_1 \ldots u_n$.

Assume that the inequality holds for some $n \geq 0$. The interval $\lin 0 ; v_{n+1} \irin$ is divided into $u_{n+1}$ sub-intervals of size $v_n$. Suppose that each of those intervals contains a big cluster of size at least $\gamma_n v_n$, except maybe one of them  which we call the \emph{bad interval}. As we already mentioned two clusters of size at least $\gamma_n v_n$ belonging to two neighboring intervals will merge since $\gamma_n > 3/4$. This means that all the big clusters belonging to sub-intervals on the left hand side of the bad interval will merge together into a cluster denoted by $C_l$. The same thing happens for the clusters on the right hand side of the bad interval creating a big cluster $C_r$. Now, the only case where $C_l$ and $C_r$ do not merge together is when the influence radius of one of them  cannot "reach over" the bad interval. This happens only if the bad interval is located at  $\lin v_n; 2 v_n\irin$ or $\lin (u_{n+1} - 2 )v_n ; (u_{n+1} - 1 ) v_n\irin$ (see Figure \ref{fig:multiplegood1D} for an illustration). In any case, at least $u_{n+1} - 2$ sub-intervals containing a big cluster merge together. Thus, the CMP on $\lin 0 ; v_{n+1} \irin$ contains a cluster of size at least
\[
(u_{n+1} - 2) \gamma_n v_n = \left(1 - \frac{2}{u_{n+1}} \right) \gamma_n v_{n+1} = \gamma_{n+1} v_{n+1}.
\]
\begin{figure}[!t]
\begin{center}
\includegraphics[width=9cm]{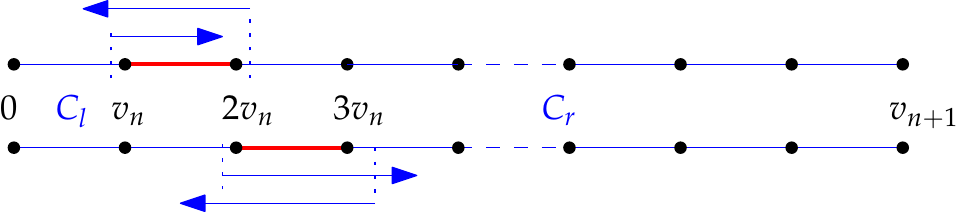}
\caption{\label{fig:multiplegood1D}Red intervals represent bad events. Top: The bad event is on the interval $\lin v_n ; 2 v_n\irin$ and $C_l$ cannot reach over to merge with $C_r$. Bottom: The bad interval is inside $\lin 2v_n ; v_{n+1} - 2v_n \irin$ and $C_l$ and $C_r$ merge together.}
\end{center}
\end{figure}
Recalling that this occurs whenever there is strictly less than two bad intervals and using the fact that the weights $(r(x))$ are i.i.d., we get, using the recurrence hypothesis,
\begin{align*}
1 - \P \left\{\kE(v_{n+1},\gamma_{n+1}) \right\} & \leq \P\big\{ \text{Binomial} \left( u_{n+1} , 1 - \P \left\{\kE(v_{n},\gamma_{n}) \right\} \right)  \geq 2 \big\}\\
& \leq u^2_{n+1}\left( 1 - \P \left\{ \kE(v_{n},\gamma_{n}) \right\} \right)^2\\
& \leq \frac{ 40^2 \varepsilon^2}{4^n} < \frac{\varepsilon}{4^{n+1}}
\end{align*}
provided that $\varepsilon$ is chosen small enough. Thus, we have proved that, for $\varepsilon >0$ arbitrarily small, we can always find $0< p <1$ such that
\[
\P \left\{ \kE(v_n,\gamma_n) \right\} > 1 - \varepsilon / 4^n \quad\hbox{for all $n\geq 0$.}
\]
We now introduce the "anchored" good events
\[
\kE_0 (n,\gamma) \defeq \left\{ \text{in $\cmp(\lin 0;n \irin)$ the cluster containing vertex $0$ has at least $\gamma n$ elements} \right\}.
\]
Using similar arguments as before, we see that the event $\kE_0 (v_{n+1},\gamma_{n+1})$ is realized whenever the following conditions are met:
\begin{itemize}
\item $\kE_0 (v_{n},\gamma_n)$ is realized;
\item the interval $\lin v_n ; 2 v_n \irin$ contains a cluster of size at least $\gamma_n v_n$;
\item For every $2 \leq k < u_{n+1}$ except maybe one, the interval $\lin k v_n ; (k+1) v_n \irin$ contains a cluster of size at least $\gamma_n v_n$.
\end{itemize}
This implies
\begin{eqnarray*}
\P \left\{ \kE_0\left(v_{n+1},\gamma_{n+1}\right) \right\} &\geq& \P \left\{\kE_0\left(v_{n},\gamma_n\right) \right\} \left(1 - \frac{\varepsilon}{4^n} \right) \P\left\{\text{Binomial} \left( u_{n+1} - 2, \frac{\varepsilon}{4^n}\right)  < 2 \right\}\\
&\geq & \P \left\{\kE_0\left(v_{n},\gamma_n\right) \right\} \left(1 - \frac{\varepsilon}{4^n} \right) \left( 1 - \frac{\varepsilon}{4^{n+1}} \right),
\end{eqnarray*}
therefore, since $\gamma_n > 3/4$ for  all $n$, we get that
\[
\liminf_{n \to \infty} \P \left\{ \kE_0(v_{n},3/4) \right\} > 0
\]
which proves that, with positive probability, the cluster $\cmp_0$ is infinite whenever $p$ is sufficiently close to $1$.
\end{proof}

\bigskip

As we will see in the next section, for any $\alpha \geq 1$, we have $p_c(\alpha) > 0$ for Bernoulli CMP on any $d$-dimensional lattice, $d\geq 1$. For the time being, we provide an alternative proof which works only in the one-dimensional case and for $\alpha  =1$ but has the advantage of providing an explicit lower bound for the critical percolation parameter.
\begin{prop}
\label{prop:BerCMPp_c>0}
Consider Bernoulli CMP on $\Z$ with $\alpha = 1$. We have
$$
p_c(1) \;\geq\; \frac{1}{2}.
$$
\end{prop}

The proof is based on the following duality lemma.

\begin{lemm}
\label{lem:dualityZ} We assume here that $\alpha = 1$. Set $I_n \defeq \lin 0; n \rin$ and fix  a sequence of weights $(r(x),x\in I_n) \in \{0,1\}^{I_n}$. Suppose that the CMP $\cmp(I_n,r)$ inside $I_n$ is such that site $0$ and site $n$ belong to the same cluster. Then, the interval $I_n$ (seen as a subset of $\Z$) is stable for the CMP $\widetilde{\cmp}(I_n,\tilde{r})$ constructed with the reversed
weights $\tilde{r}(x) \defeq 1 - r(x)$.
\end{lemm}
\begin{proof}[Proof of Lemma \ref{lem:dualityZ}.]
The proof goes by induction. The case $n=1$ is trivial. Suppose now that the result holds for every $k \leq n$. We consider $I_{n+1}$ with a sequence of weights satisfying the assumptions of the lemma. Since sites $0$ and $n+1$ are in the same cluster, we have in particular $r(0) = r(n+1) = 1$. According to the cluster merging procedure, the partition $\cmp(I_{n+1})$ is obtained by successive application of the merging operator. We remark that there are at most $n$ true merging needed to construct $\cmp(I_{n+1})$ (because the number of clusters decreases by one after each true merging). Let $(x_k,y_k)$ be a sequence of pairs of sites describing a merging history for $\cmp(I_{n+1})$ \emph{i.e.}
$$
\cmp(I_{n+1}) = M_{x_m,y_m} \circ M_{x_{m-1},y_{m-1}} \circ \cdots \circ M_{x_1,y_1} (I_{n+1})
$$
(where we identify $I_{n+1}$ with its trivial partition). Let $k\leq m$ denote the index where the operator $M_{x_k,y_k}$ merges together the cluster containing $0$ with the cluster containing $n+1$ and let $\kC$ be the partition of $I_{n+1}$ obtained just before this merging occurs. We
 set
$$
L \defeq \max \kC_0 \quad \hbox{ and } R \defeq \min \kC_{n+1}
$$
\begin{figure}[!t]
\begin{center}
\includegraphics[width=9cm]{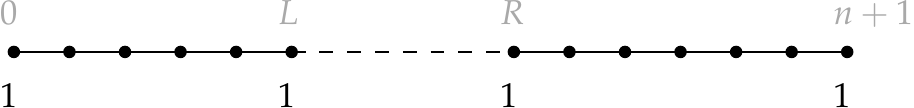}
\caption{\label{fig:duality}Illustration of the proof of Proposition \ref{prop:locality}. In this configuration, $0$ and $L$ are both in a cluster composed of vertices of $\lin 0;L \rin$. Vertices $R$ and $n+1$ are also both in a cluster composed of vertices of $\lin R ; n+1 \rin$. In the next step of the cluster merging procedure, these two clusters merge.}
\end{center}
\end{figure}
\emph{i.e.} $L$ is the rightmost site of the cluster of $\kC$ containing $0$ and $R$ is the leftmost site of the cluster containing $n+1$. See Figure \ref{fig:duality} for an illustration of this configuration. Since these two clusters can merge and since the weight of a cluster is equal to its number of sites (because each weight is either $0$ or $1$), we have
\begin{equation}\label{dualityeq1}
R - L \leq \min( L+1 , n+2-R).
\end{equation}
Furthermore, we observe that all the merging used to construct the cluster $\kC_0$ containing $0$ and $L$ are also valid when considering the CMP inside $\lin 0 ; L \rin$. This means that the CMP on $\lin 0 ; L \rin$ has a cluster containing both $0$ and $L$. Similarly, the CMP inside $\lin R ; n+1 \rin$ has a cluster containing both $R$ and $n+1$. Applying the induction hypothesis, we deduce that the subsets $\lin 0; L \rin$ and $\lin R; n+1 \rin$ are stable for the reversed weight sequence $\tilde{r}$. Moreover, the sum of the weights $\tilde{r}$ of all the sites inside the center interval $\ilin L ; R \irin$ is bounded above by $R - L - 1$. Combining this with \eqref{dualityeq1}, we find that
$$
\sum_{x\in \ilin L; R\irin} \tilde{r}(x)\leq \min( L, n+1-R) < d(\ilin L; R\irin \;,\; \Z\setminus  \lin 0; n+1 \rin)
$$
and we conclude using Proposition \ref{prop:etouffement} that  $\lin 0 ; n+1 \rin$ is indeed stable set for the reversed weight sequence $\tilde{r}$.
\end{proof}

\begin{proof}[Proof of Proposition \ref{prop:BerCMPp_c>0}.]
By symmetry, the probability of having an infinite cluster unbounded towards $+\infty$ is the same as that of having an infinite cluster unbounded toward $-\infty$. By translation invariance, these probability are either $0$ or $1$. Thus, if the infinite cluster exists, it is a.s. unbounded towards both $+ \infty$ and $- \infty$.

We fix $p=\frac{1}{2}$ and we suppose  by contradiction that $p_c < p$. Thus, $\cmp$ contains an infinite cluster a.s. which we denote by $C^\infty$. Let $(x_n,y_n)_{n\geq 0}$ be a (random) sequence of pairs of vertices associated with the cluster merging procedure \emph{i.e.} such that $\lim_n\uparrow \kC^n = \cmp$ where
$$\kC^n \defeq M_{x_n,y_n} \circ \cdots \circ M_{x_0,y_0} (\Z)$$
(we identify $\Z$ with its finest partition). We can find an increasing function $\varphi$ such that $x_{\varphi(n)}$ and $y_{\varphi(n)}$ both belong to $C^\infty$ for every $n$. The sequence of sets $C^n \defeq \kC^{\phi(n)}_{x_{\phi(n)}}$ increases weakly to $C^\infty$ as $n$ tend to infinity. Let $L_n \defeq \min C_n$ and $R_n \defeq \max C_n$ be the leftmost and rightmost vertices of $C^n$. The cluster $C^\infty$ being unbounded in both directions, $L_n$ and $R_n$ diverge respectively go to $-\infty$ and $+ \infty$ as $n \to \infty$. Moreover, by construction $L_n$ and $R_n$ are in the same cluster if we consider the CMP inside $\lin L_n ; R_n \rin$. Thus, according to Lemma \ref{lem:dualityZ}, the intervals $\lin L_n ; R_n \rin$ are stable sets for the reversed weight sequence. Moreover, the union of these sets exhausts $\Z$ so Theorem \ref{theo:equivInfiniteCluster} states that the CMP on $\Z$ with weight sequence $\tilde{r}$  contains only finite clusters. But since $p=\frac{1}{2}$, the sequences $(r(x),x\in Z)$ and $(\tilde{r}(x),x\in Z)$ have the same law which leads to a contradiction.
\end{proof}
\begin{samepage}
\begin{rema} We do not believe the lower bound $1/2$ of Proposition \ref{prop:BerCMPp_c>0} to give the critical value for Bernoulli CMP on $\Z$. In fact, numerical simulations suggest that
$$
p_c(1) \simeq 0.65.
$$
We ask the question: does there exist an explicit formula for this critical parameter?
\end{rema}
\end{samepage}

\subsection{Phase transition on \texorpdfstring{$\Z^d$}{Zd} for continuum CMP.}
\label{sec:Zdiid}
In the previous section we have shown that, on any infinite graph, there exists a super-critical phase where an infinite cluster is present. In this section we prove that, on finite dimensional lattices, there also exists a sub-critical phase where every cluster is finite.
\begin{prop}
\label{prop:CMPsubcritical}
Let $d\geq 1$.  Consider continuum CMP on $G = \Z^d$ with expansion exponent $\alpha \geq 1$  and i.i.d. weights
distributed as $\lambda Z$. Suppose that $\E[Z^{\beta_0}] < +\infty$ for $\beta_0 \defeq (4\alpha d)^2$. Then, we have
$$
\lambda_c(\alpha) > 0.
$$
\end{prop}

By coupling, this proposition implies that the critical parameter for Bernoulli percolation is strictly positive in any dimension:

\begin{coro}
\label{coro:BerCMPsubcritical}
Consider Bernoulli CMP  on $\Z^d$ with expansion exponent $\alpha \geq 1$. We have
$$p_c(\alpha) > 0.$$
\end{coro}

\begin{rema}
The moment condition is obviously not optimal. We conjecture that a sub-critical phase exists whenever $Z$ admits a moment of order $\alpha d + \varepsilon$ for $\varepsilon$ small enough. Conversely, if $Z$ does not have moments of order $\alpha d - \varepsilon$, then
the maximum influence radius of sites at distance $N$ from the origin is much larger than $N$ and one can show that the CMP always contains an infinite cluster.
\end{rema}

\begin{proof}[Proof of Proposition \ref{prop:CMPsubcritical}.]
The proof is again based on multiscale analysis. We define the ``good event'':
\[
\kE (N) \defeq \left\{\text{there exists a stable set $S$ such that } \, \lin N/5 \; ;\; 4N/5 \rin^d \subset S \subset \lin 1\; ;\; N \rin^d\right\}.
\]
We will show that, when $\lambda$ is small enough, events $\kE(N)$ occur for infinitely many $N$ almost surely.
\begin{figure}[!t]
\begin{center}
\includegraphics[width=14cm]{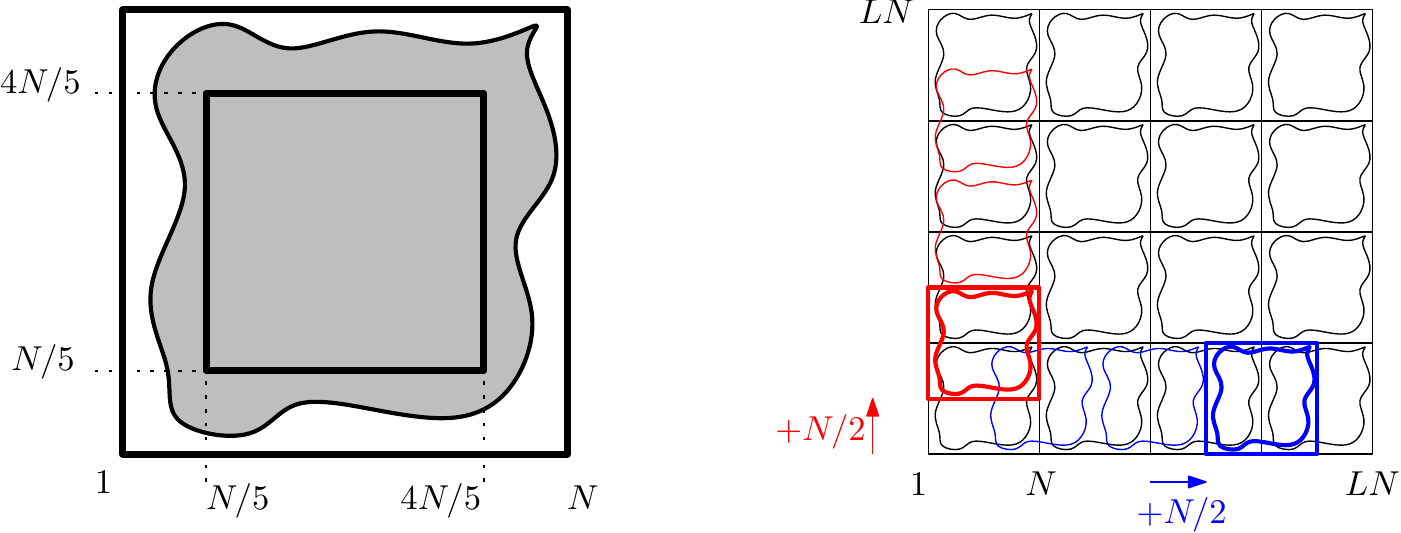}
\caption{\label{fig:goodevent}Left: An event $\kE(N)$, the grey set is stable. Right: Constructing $\kE(LN)$ from $\kE(N)$.}
\end{center}
\end{figure}

First, we remark that, since a union of stable sets is itself stable, for $N,L \in \N$, the event $\kE(LN)$ can be constructed by translating boxes of side length $N$ inside which $\kE(N)$ occurs. In order to cover the gaps on the boundaries of the smaller boxes, we cover the larger box using translations by $N/2$ for each coordinate, accounting for a total of $(2L-1)^d$ boxes of volume $N^d$. See Figure \ref{fig:goodevent} for an illustration. Of course, not every one of these $(2L -1)^d$ boxes covering $\lin 1 ; LN \rin^d$ needs to satisfy $\kE(N)$ in order for $\kE(LN)$ to be realized. We will say that the boxes for which $\kE(N)$ does not hold are ``bad boxes''. For $k \geq 0$, we define the event:
\[
\kE_k(N,L) \defeq \left\{
\begin{gathered}
\text{at most $k$ of the $(2L -1)^d$ boxes of side length $N$}\\
\text{covering $\lin 1 ; LN \rin^d$ are bad boxes}
\end{gathered}
\right\}.
\]
We focus on events of the form $\kE(R_n)$ and $\kE_k(R_n,L_{n+1})$ where
\begin{equation*}
L_n \defeq 2^{c^n} \quad\hbox{and}\quad R_n \defeq L_1 \ldots L_n
\end{equation*}
with $c \defeq 2\alpha d + 1$. We remark that
\begin{equation}\label{eqRnLn}
R_n \; = \; 2^{c+ c^2 + \ldots + c^n} \; = \;  2^{\frac{c^{n+1} -1}{c-1}-1} \; \leq \;  L_{n+1}^{\frac{1}{c-1}}.
\end{equation}
We also define
\begin{equation*}
\varepsilon_n \defeq 2^{- 2 d c^{n+1}}.
\end{equation*}
The next lemma estimates the number of bad boxes among the boxes of side length $R_n$ used to cover the larger box of side length $R_{n+1} = L_{n+1}R_n$.
\begin{lemm}
\label{lem:nbbadboxes}
Set $k_0 \defeq \lceil 2^{d+1}(c + 1) \rceil$. Suppose that for some $n \geq 1$, we have
$$\P \left\{ \kE \left(R_n \right) \right\} \geq 1 - \varepsilon_n.$$
Then, it holds that
\[
\P \left\{ \kE_{k_0} (R_{n},L_{n+1}) \right\} \geq 1 - \frac{1}{2} \varepsilon_{n+1}.
\]
\end{lemm}
\begin{proof}
The box $\lin 1 ; R_{n+1} \rin^d$ is covered by $(2L_{n+1}-1)^d$ smaller boxes of side length  $R_n$ shifted by multiples of $R_n/2$. These boxes intersect each other so the individual events that they are ``good boxes'' are not independent. However, we can partition this set of boxes into $2^d$ groups such that each group contains at most $L_{n+1}^d$ mutually disjoint boxes (\emph{i.e.} we put in the same group boxes translated by multiples of $R_n$ in each direction). Therefore, if the event $\kE_{k_0}(R_n,L_{n+1})$ fails, one of these $2^d$ groups has to contain at least $k'\defeq \lceil \frac{k_0+1}{2^d} \rceil \geq 2c + 2$ bad boxes. Using the fact that the events of being bad boxes are independent within each group of boxes, we find that
\begin{eqnarray*}
1 - \P \left\{ \kE_{k_0} (R_{n},L_{n+1}) \right\} & \leq  & 2^d \P \left\{ \hbox{Binom} \left(L_{n+1}^d, \varepsilon_n  \right) \geq k' \right\}\\
& \leq & 2^d \left( L_{n+1}^d  \varepsilon_n \right)^{k'}\\
& = & \frac{1}{2} 2^{d + 1 + d k'  c^{n+1} -2 d k' c^{n+1}}\\
& = & \frac{1}{2} 2^{d + 1 - k' d c^{n+1}  + 2d c^{n+2}}\; \varepsilon_{n+1}\\
& \leq & \frac{1}{2}\varepsilon_{n+1}
\end{eqnarray*}
where we used that $d + 1 - k' d c^{n+1}  + 2d c^{n+2} \leq 0$ when $k' \geq 2c + 2$ for the last inequality.
\end{proof}
We introduce another family of events meant to control to total sum of the influence radii inside a box:
$$
\kA(N,L) \defeq \Big\{ \sum_{x\in \lin 1 ; N \rin^d }r(x) \leq L^{\frac{1}{\alpha}} \Big\}
$$
and we set
$$
\kG(N,L) \defeq \left\{
\begin{gathered}
\hbox{each of the $(2L-1)^d$ sub-boxes of side length $N$}\\
\hbox{covering the box $\lin 1 ; NL \rin^d$ satisfies $\kA(N,L)$}
\end{gathered}
\right\} .
$$
\begin{lemm}\label{lem:maxrad} There exists $n_0$ such that for all $n\geq n_0$, uniformly on $\lambda \leq 1$, we have
$$
\P\{\kG(R_{n},L_{n+1})\}
\geq 1 - \frac{1}{2}\varepsilon_{n+1}
$$
\end{lemm}
\begin{proof} The proof uses only crude estimates and union bound. First, when $\kA(N,L)$ fails, then there is at least one site $x \in \lin 1 ; N \rin^d$ such that $r(x) \geq \frac{L^{1/\alpha}}{N^d}$. Consequently, we have
\begin{eqnarray*}
1 - \P\{ \kA(R_n,L_{n+1}) \} &\leq& R_n^d \P\Big\{\lambda Z \geq \frac{(L_{n+1})^{\frac{1}{\alpha}}}{R_n^d} \Big\} \\
&\leq & \frac{\E[Z^{\beta_0}] R_n^{d(1 + \beta_0)}}{(L_{n+1})^{\frac{\beta_0}{\alpha}}}\\
&\leq& \E[Z^{\beta_0}](L_{n+1})^{\frac{d(1 + \beta_0)}{c-1}-\frac{\beta_0}{\alpha}}
\end{eqnarray*}
where we used Markov's inequality and $\lambda \leq 1$ for the second inequality and \eqref{eqRnLn} for the last one. Now, using again the union bound, we find that
\begin{eqnarray*}
1 - \P\{\kG(R_{n},L_{n+1})\}
&\leq& (2L_{n+1}-1)^d \left(1 - \P\{ \kA(R_n,L_{n+1})\}\right)\\
 &\leq& 2^d \E[Z^{\beta_0}]
 (L_{n+1})^{d + \frac{d(1 + \beta_0)}{c-1}-\frac{\beta_0}{\alpha}}\\
 &=& 2^d \E[Z^{\beta_0}] \varepsilon_{n+1}
(L_{n+1})^{d + \frac{d(1 + \beta_0)}{c-1}-\frac{\beta_0}{\alpha} + 2dc}.
\end{eqnarray*}
Recalling that $c \defeq 2\alpha d + 1$, the exponent of $L_{n+1}$ in the formula above is equal to
$$
3d + \frac{1}{2\alpha} + 4\alpha d^2 - \frac{\beta_0}{2\alpha} \;\leq\; 8\alpha d^2 - \frac{\beta_0}{2\alpha} \;<\; 0
$$
which completes the proof of the lemma.
\end{proof}

\begin{lemm}
\label{lem:radbadboxes} Let $k_0$ be as in Lemma \ref{lem:nbbadboxes}. There exists $n_1 > n_0$ such that for all $n \geq n_1$, uniformly in $\lambda \leq 1$, it holds that
\begin{equation*}
\kE_{k_0} (R_{n},L_{n+1}) \cap \kG(R_{n},L_{n+1}) \subset \kE(R_{n+1}).
\end{equation*}
\end{lemm}
\begin{proof}
The idea behind this inclusion is the following: on $\kE_{k_0} (R_{n},L_{n+1})$, there are at most $k_0$ sub-boxes which do not contain a stable set. On the other hand, on  $\kG(R_{n},L_{n+1})$, the influence radius of these non stable sets is negligible compared to the diameter of the big box. Therefore, either a bad box is close to the boundary and it does not interfere with a stable set in the center of the box or it is at a macroscopic distance from the boundary in which case it is contained in a stable set according to Proposition \ref{prop:etouffement}.

More precisely, suppose that we are on the event $\kE_{k_0} (R_{n},L_{n+1}) \cap \kG(R_{n},L_{n+1})$. Let  $k \leq k_0$ be the number of bad boxes. We denote them by $B_1,\ldots, B_k$ and set
$$
R \defeq \sum_{x \in \bigcup_i B_i}r(x).
$$
Let $n\geq n_0$. Since we are on $\kG(R_{n},L_{n+1})$, Jensen's inequality gives
\begin{equation*}
R^\alpha \;\leq\; k^{\alpha-1}\sum_{i=1}^k \left(\sum_{x \in B_i}r(x)\right)^\alpha \;\leq\; k_0^\alpha L_{n+1}.
\end{equation*}
\begin{figure}[!t]
\begin{center}
\includegraphics[width=8cm]{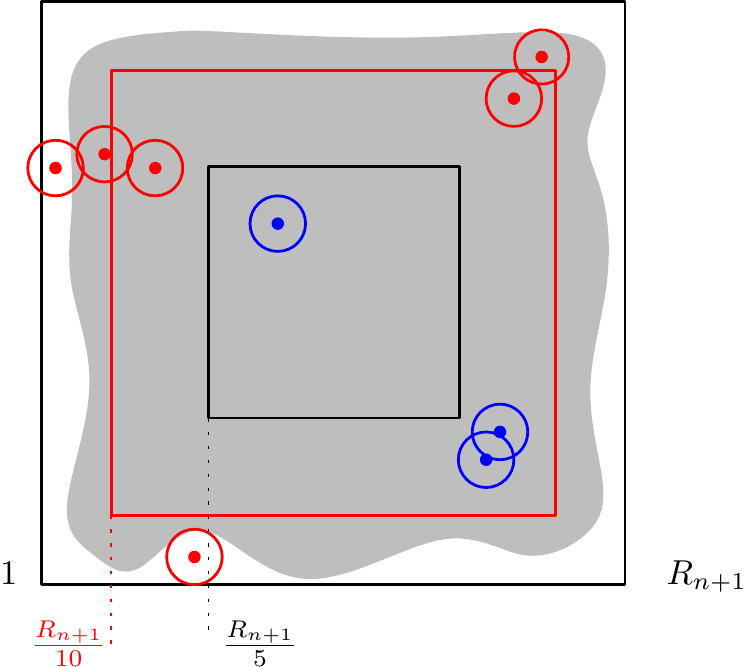}
\caption{\label{fig:badboxes}In grey: the union of all the stable sets inside good sub-boxes. In red: the bad boxes $B_i$ and their neighborhoods $\widetilde{B}_i$ which are connected to the outside. In blue: the bad boxes $B_j$ and their neighbourhoods $\widetilde{B}_j$ which are not connected to the outside.The union of the grey set and the blue bad boxes is a stable set which contains the inner region delimited by the black square.}
\end{center}
\end{figure}
We consider a neighbourhood around each of the boxes of the form:
$$
\widetilde{B}_i \defeq \{ x\in\Z^d\, : _, d(x,B_i) \leq R^\alpha\}.
$$
From the triangle inequality, it comes that $\diam(\widetilde{B}_i) \leq 2 R^\alpha + R_n \leq 3 k_0^\alpha L_{n+1}$. Consequently, for $n$ large enough,
\begin{equation}\label{connectedout}
\sum_{i=1}^{k} \diam(\widetilde{B}_i) \;\leq \; 3k_0^{\alpha+1} L_{n+1} \; < \; \frac{R_{n+1}}{20}.
\end{equation}
We say that a bad box $B_i$ is connected to the outside if there exists a path from some vertex of $B_i$ to some vertex outside of the box $\lin R_{n+1}/10\; ; \; 9R_{n+1}/10 \rin^d$ which stays inside $\bigcup_{j}\widetilde{B}_j$.

Let $S_0$ be the union of all the stable sets of the good sub-boxes of $\lin 1 ; R_{n+1}\rin^d$. Let $S_1$ be the union of $S_0$ together with all the bad boxes that are not connected to the outside. See Figure \ref{fig:badboxes} for an illustration. The set $S_0$ is stable as a union of stable sets. Moreover, by definition $W \defeq S_1 \setminus S_0$ is included in the union of all the bad boxes hence
\begin{equation}\label{boots1}
\sum_{x\in W} r(x) \leq R.
\end{equation}
Since we added in $S_1$ only the bad boxes which are not connected to the outside, we have
\begin{equation}\label{boots2}
B(W,R^\alpha) \subset S_1.
\end{equation}
In view of Proposition \ref{prop:etouffement}, inequality \eqref{boots1} combined with \eqref{boots2} implies that $S_1$ is also stable. Finally, inequality \eqref{connectedout} says that any bad box connected to the outside is at distance at most $\frac{R_{n+1}}{20}$ from the outside of the box $\lin R_{n+1}/10 \; ; \; 9R_{n+1}/10 \rin^d$. This implies
$$
S_1 \supset \lin R_{n+1}/5 \; ; \; 4R_{n+1}/5\rin^d
$$
which completes the proof of the lemma.
\end{proof}

We can now finish the proof of the proposition. As $\lambda$ tend to zero, we have $\lambda r(x)\to 0$ for all $x\in\Z^d$. Since $\kE \left(R_{n_1}\right)$ depends only the values of $r$ for finitely many sites, it follows that
$$
\lim_{\lambda\to 0} \P\left\{ \kE \left(R_{n_1} \right) \right\} = 1
$$
so we can fix $\lambda >0$ such that
$$
\P\left\{\kE \left(R_{n_1} \right)\right\} \geq 1 - \varepsilon_{n_1}.
$$
We prove by induction that the same inequality holds for all $n\geq n_1$. Indeed, if it holds for $n$, then, combining Lemmas \ref{lem:nbbadboxes},\ref{lem:maxrad} and \ref{lem:radbadboxes}, we find that
\begin{equation*}
\P\left\{\kE \left(R_{n+1}  \right)\right\} \;\geq\; 1 -
\left(1-\P\left\{\kE_{k_0} (R_{n},L_{n+1})\right\}\right)
- \left(1-\P\left\{\kG(R_{n},L_{n+1})\right\}\right)
\;\geq\; 1 - \varepsilon_{n+1}.
\end{equation*}
By translation invariance and the Borel-Cantelli lemma we conclude that
$$
\P\left\{ \hbox{there exists a finite stable set containing $\lin -N ; N \rin^d$} \right\} = 1
$$
for every $N$ which implies that the CMP has no infinite cluster.
\end{proof}

\subsection{Phase transition on random geometric graphs and Delaunay triangulations}\label{sec:RGG and Delaunay}

We now explain how to extend Proposition \ref{prop:CMPsubcritical} for the degree-weighted CMP (Model \ref{DegCMP}) when the graph $G$ is either a random geometric graph or a Delaunay triangulation constructed from a Poisson point process $\kP$ in $\R^d$ with Lebesgue intensity. First, we quickly recall the definition of these graphs:

\bigskip

$\bullet$ \emph{Geometric graph with parameter $R > 0$.} The vertex set is composed of the atoms of the point process $\kP$ and, for any pair of points $x,y \in \kP$, there is an edge between $x$ and $y$ if and only if $\|x-y \| < R$, where $\| \cdot \|$ denotes the Euclidian norm in $\R^d$. If $R$ is above the critical parameter for continuum percolation, then this graph has a unique infinite connected component (see for instance \cite{MR96}). We assume this is the case and denote this graph $\mathcal{G}(R,\mathcal{P})$.

\bigskip

$\bullet$ \emph{Delaunay Triangulation.} For any $x\in \kP$, we define the Voronoï cell of $x$ as the set of points of $\R^d$ which are closer to $x$ than to any other point of the Poisson point process:
\[
\mathrm{Vor}_{\kP}(x) \defeq \left\{ z \in \R^d : \|x-z\| < \| y-z\| \, \forall y \in \kP \right\}.
\]
The Delaunay triangulation of $\kP$ is the dual of the Voronoï tessellation: its vertex set is again the set of atoms of $\kP$ and two vertices share an edge if and only if their corresponding cells are adjacent \emph({i.e.} they share a $d-1$ dimensional face). We denote this graph by $\mathcal{D}(\kP)$. Since the points of the Poisson process are almost surely in general position in $\R^d$, this triangulation is also characterized by the following property: for any simplex of $\mathcal{D}(\kP)$, its circumscribed sphere contains no point of $\kP$ in its interior.

\begin{prop}
\label{prop:CMPDelaunay}
Consider the CMP on $\mathcal{G}(R,\kP)$ or on $\mathcal{D}(\kP)$ with weights given by
$$
r(x) \defeq \hbox{\textup{deg}}(x)\ind_{\{\hbox{\textup{deg}}(x) \geq \Delta\}}
$$
and with expansion exponent $\alpha \geq 1$
Then, we have
\[
\Delta_c (\alpha) <\infty.
\]
\end{prop}
The proof is very similar to the proof of Proposition \ref{prop:CMPsubcritical}. It would be redundant to write everything again so we will simply point out the modifications needed to adapt the proof for these random graphs. The main difference (and difficulty) in our new setting  comes from the fact that two portions of the graph included in disjoint domains of $\R^d$ are not independent anymore. The modification needed for the geometric graph are minor since sub-graphs included in domains separated by a distance larger than $R$ are still independent. The situation is a little more complex for Delaunay triangulations so we will concentrate on this case. Details for the geometric graph are left to the reader.

\begin{proof}[Sketch of the proof]
Recall the parameters of the proof of Proposition \ref{prop:CMPsubcritical}:
\begin{equation*}
R_n \defeq L_1 \ldots L_n;  \quad L_n \defeq 2^{c^n}; \quad \varepsilon_n \defeq 2^{- 2 d c^{n+1}}
\end{equation*}
for $c$ suitably chosen and depending only on $\alpha$ and $d$.
We adapt the definition of good events to the new setting:
\[
\kE (R_n) \defeq \left\{
\begin{gathered}
\text{there exists a stable set $S$ such that } \\
\left( \mathcal{D}(\kP) \cap [ R_n/5 , 4 R_n /5 ]^d \right)
\subset S \subset
\left( \mathcal{D}(\kP) \cap [ 0, R_n ]^d \right)
\end{gathered}
\right\}.
\]
We want to prove by induction that for $n$ large enough
\[
\P \left\{ \kE (R_n) \right\} \geq 1 - \varepsilon_n.
\]
Once we have proved that the three lemmas \ref{lem:nbbadboxes}, \ref{lem:maxrad} and \ref{lem:radbadboxes} hold for our random graph, the result follows \emph{mutatis mutandis}.

\bigskip

Lemma \ref{lem:radbadboxes} is a statement concerning stable sets of the CMP and it does not really depend on the particular nature of the graph. It is straightforward to adapt it to the case of Delaunay triangulations and other random graphs embedded in $\R^d$.

\bigskip

Lemma \ref{lem:maxrad} is also easily translated to our case. This lemma gives an estimate for the probability  that every box has a reasonable total radius. As we already pointed out, it uses only the union bound so it does not require any kind of independence assumption. It also requires that the weights admit polynomial moments of high enough order. This is not a problem here since
the typical distribution of the degree of a site for the geometric graph and the Delaunay triangulation has, in fact, exponential moments: this is straightforward for $\mathcal{G}(R,\kP)$ since the degree of a site is bounded by the number of  atoms of $\kP$ inside a ball of radius $R$. For $\mathcal{D}(\kP)$, this result is proved in \cite{Z} (and it also follows from Lemma \ref{lemm:delaunay} which we will prove later on). The last point to check is that we can upper bound the number of vertices inside a box. For the lattice case $\Z^d$, this number was deterministic and equal to the volume of the box. In our new setting, it is still  straightforward since the number of vertices is simply the number of atoms of $\kP$ and hence it follows a Poisson distribution with parameter equal to the volume of the box. In particular, this distribution has light tails which is more than we need.

\bigskip

Lemma \ref{lem:nbbadboxes} requires a bit more work. It controls the number of bad boxes in $[0, R_{n+1}]^d$. Recall that it states that, when
\[
\P \left\{ \kE \left(R_n\right) \right\} \geq 1 - \varepsilon_n,
\]
then it holds that
\[
\P \left\{ \kE_{k} (R_{n},L_{n+1}) \right\} \geq 1 - \frac{1}{2} \varepsilon_{n+1}
\]
where
\[
\kE_k(R_{n},L_{n+1}) \defeq \left\{
\begin{gathered}
\text{at most $k$ of the boxes of side length $R_n$}\\
\text{covering $[0,R_{n+1}]^d$ are bad boxes}
\end{gathered}
\right\}.
\]
For the lattice $\Z^d$, we proved this by partitioning the set of small boxes covering $[0,R_{n+1}]^d$ into groups containing disjoint boxes. Then we used the fact that good events were independent for disjoint boxes to compare the number of bad boxes with a binomial distribution. We cannot do this directly now since we do not have independence of events occurring in disjoint boxes any more. This problem is easy to overcome for the random geometric graph as we simply partition the set of boxes into more groups in such way that two boxes in the same group are at distance at least $R$. The adaptation of the lemma in the case of the Delaunay triangulation is a bit more involved. Consider the events
\[
\mathcal{I}(N,\eta) \defeq \left\{
\begin{gathered}
\text{There are no edges linking vertices of $\kP \cap [\eta N, (1- \eta) N]^d$}\\
\text{with vertices of $\kP \cap (\R^d\setminus [0,N]^d)$}
\end{gathered}
 \right\}.
\]

\begin{lemm}
\label{lemm:delaunay}
Let $\eta <1/5$. There is an event $ \widetilde{\mathcal{I}}(N,\eta)$ depending only on the points of $\kP$ inside the annulus $A_\eta \defeq [0,N]^d \setminus [\eta N, (1- \eta) N]^d$ with the following properties:
\begin{gather*}
\widetilde{\mathcal{I}}(N,\eta) \subset \mathcal{I}(N,\eta),\\
\P \left\{ \widetilde{\mathcal{I}}(N,\eta) \right\} \geq 1 - C e^{- C N^d}
\end{gather*}
where $C>0$ is a constant depending only on $\eta$ and $d$.
\end{lemm}
Let us first explain how we use this result to get the desired estimate. Let $\vartheta_n$ denote the probability that a small box does not satisfy $\widetilde{I}(R_n,\eta)$. By union bound, the probability that one of the boxes covering the larger box does not satisfies $\tilde{I}(R_n,\eta)$ is smaller than $\vartheta_n (2 L_{n+1} - 1)^d = o(\varepsilon_{n+1})$. Thus, we condition on the event that every sub-boxes satisfy $\tilde{I}(R_n,\eta)$. Then, the events when the small boxes are bad become independent for disjoint boxes and have a probability uniformly smaller than $\varepsilon'_n = \varepsilon_n + \vartheta_n \sim \varepsilon_n$. Thus, we can again use a comparison with a binomial random variable with probability of success $1 - \varepsilon'_n$ and the rest of the proof is the same as in the lattice case.
\end{proof}

\begin{figure}[!t]
\begin{center}
\includegraphics[width=11cm]{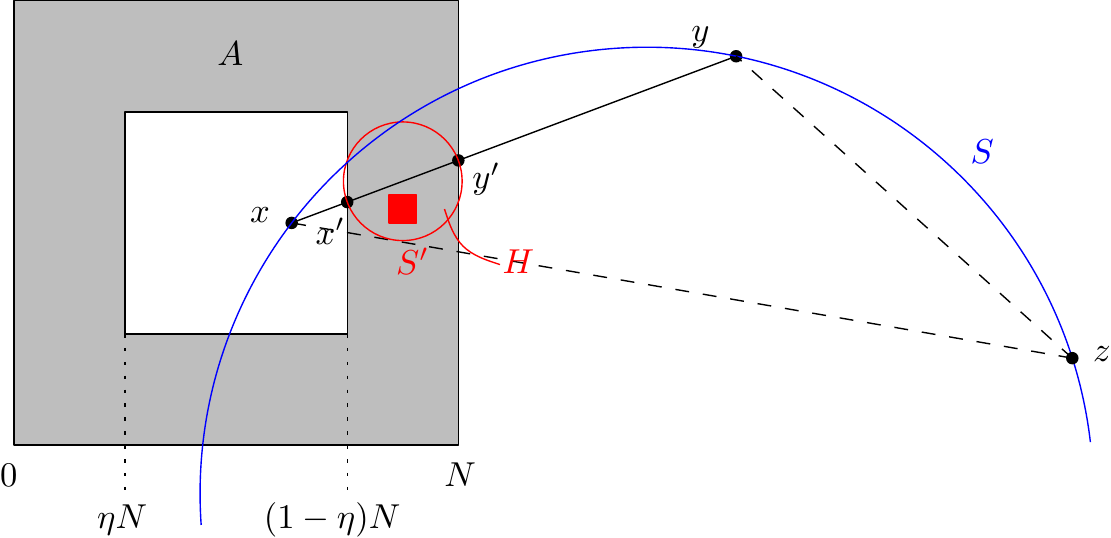}
\caption{\label{fig:geomlemma} The blue circle circumscribes  a triangle face $(x,y,z)$ of $\kD(\kP))$ hence it contains no point of $\kP$ in it interior. The half disk $H$ contains a red box of macroscopic size included in the annulus $A$.}
\end{center}
\end{figure}

\begin{proof}[Proof of Lemma \ref{lemm:delaunay}]
Take $ x \in \kP \cap [\eta N, (1- \eta) N]^d$ and  $y \in \kP \cap (\R^d\setminus [0,N]^d)$ and suppose that there is an edge between $x$ and $y$ in $\mathcal{D}(\mathcal{P})$. There exists an (hyper)sphere $S$ circumscribing a simplex of the triangulation and having the segment $[x,y]$ has a chord. By definition of the Delaunay triangulation, this sphere has no points of $\kP$ in its interior. Let $[x',y'] \defeq [x,y] \cap A_\eta$. The sphere $S'$ having the segment $[x',y']$ as a diameter has an hemisphere $H$ included in $S$. It is now easy to convince oneself that there is always a square box $B$ of side length $\delta N$ included in the intersection $A \cap H$ of the annulus $A$ and the hemisphere $H$ (\emph{c.f.} figure \ref{fig:geomlemma}). By construction this box does not contain any point of $\kP$. Moreover, $\delta$ depends on $\eta$ and the dimension $d$ but not on $N$.
Thus, we just need to construct an event $\widetilde{I}$ for the point process $\kP$ such that any possible square box of side length $\delta N$ inside $A_{\eta}$ contains at least one atom of $\kP$. This is achieved by partitioning the annulus into boxes of side length $\eta N / k$ where
$k \defeq \lceil 2\eta /\delta\rceil$ and then requesting that the point process has at least one atom in the interior of every box. Again, the number $K$ of boxes needed depend only on $\eta$ and $d$. The probability that the event $\widetilde{I}$ constructed in this way fails is the probability that there exists an empty box. By union bound, this is smaller than $K \exp(-(\eta N/k)^d)$ and the lemma follows.
\end{proof}

\begin{rema}\
\begin{itemize}
\item The multiscale technique used to prove the existence of stable sets and therefore a sub-critical phase is quite robust. Here, we only used it for three particular weighted graphs but it is easy to convince oneself that it can be applied for many other graphs that can be embedded into a finite dimensional space in a ``nice way'' (for graphs with exponential growth, using a multiscale method seems much more challenging). For example, the previous arguments work for general point processes provided that the intensity is bounded away from $0$ and $\infty$. We can also consider random radii when constructing the random geometric graph as long as the distribution has very light tails.

\item Another family of graphs that we think would be interesting to study  are the infinite uniform planar maps (such as the uniform infinite planar triangulation). These are graphs for which the degree of a typical site has exponential tails so we expect again that $\Delta_c > 0$ for degree weighted CMP. However, this will require more work and it is not clear (to us) what embedding into $\R^2$ should be chosen in order to use a multiscale argument.
\end{itemize}
\end{rema}


\section{Connection with the contact process}
\label{sec:connectioncontact}

In this section, we make rigorous the heuristic given in the introduction by relating the almost sure extinction of the contact process on an infinite graph to the existence of a sub-critical phase for degree-weighted CMP on the same graph. 

\subsection{The contact process on a locally finite graph}

Recall that $G = (V,E)$ is a locally finite and connected graph. Fix a parameter $\lambda>0$ which we call "infection rate". The contact process $\xi = (\xi(t),\, t\geq 0)$ on $G$ is a continuous time Markov process taking values in $\{0,1\}^V$ with transition rates given, for each $x\in V$ and $A\subset V$, by
\begin{subequations}
\begin{alignat}{2}
& A \to  A - \{x\} &\quad& \text{at rate $1$}, \label{eq:recover}\\
& A \to A \cup \{x\} && \text{at rate $\lambda|\{y\in A,\, d(x,y)=1 \}|$}. \label{eq:infection}
\end{alignat}
\end{subequations}
Thus $\xi(t)$ represents the subset of infected sites at time $t$ and \eqref{eq:recover} means that each infected site recovers at rate $1$ whereas \eqref{eq:infection} states that each site, while infected, emits independent infection vectors along its adjacent edges at rate $\lambda$.

When the graph $G$ has bounded degree (in particular when it is finite), classical theorems concerning interacting particle systems show that these transition rates define a unique Feller process (see for instance \cite{Li1} or \cite{Li2} for details). However, in our setting, the graph $G$ usually has unbounded degrees and we need to be a bit more careful when defining the contact process. In order to do so, we use the classical ``graphical construction'' which we briefly recall, see for example \cite{Li1}, p32 for additional details about this representation.

For each $x\in V$, let $N_x$ denote a Poisson point process with intensity $1$ on $\R_+$. For each oriented edge $(x,y)$, let  $N_{x,y}$ denote a Poisson point process on $\R_+$ with intensity $\lambda$. We assume that all these Poisson processes are independent. Consider $H = V\times \R_+$. For each $x\in V$, put "recovery" marks on the time-lines $\{x\} \times \R_+$ at the position of the atoms of $N_x$. For each oriented edge $(x,y)$, put arrows from $(x,t)$ to $(y,t)$ at the times $t$ corresponding to atoms of $N_{x,y}$. Following Liggett \cite{Li1}, we call active path a connected oriented path in $H$ which moves along the time lines in the increasing $t$ direction, jumps from a site to a neighbouring one using the oriented arrows but never crosses any recovery mark.  Then, we define the contact process $(\xi(t),t\geq 0)$ on $G$ starting from an initial infected configuration $\xi(0) = A$ in the following way:
\begin{equation}\label{defCP}
\xi(t) \defeq \{x\in V,\, \text{there exists a finite active path from $(y,0)$ to $(x,t)$ for some $y\in A$} \}
\end{equation}
(see  Figure \ref{fig:contactgeometric} for an illustration of this construction). This construction defines the contact process for all time $t\geq 0$ in terms of a particular oriented percolation process on $G\times\R_+$. Let us point out that without any additional assumption on $G$, it is possible that the process starting from a finite number of infected sites blows-up (\emph{i.e.} creates infinitely many particles) in finite time (this corresponds to having an infinite percolation cluster in a slice $G\times [0,t]$). However, we will not be concerned with this case as it will be ruled out by the additional assumptions that we shall make on the graph $G$.

\begin{figure}[!t]
\begin{center}
\includegraphics[width=11cm]{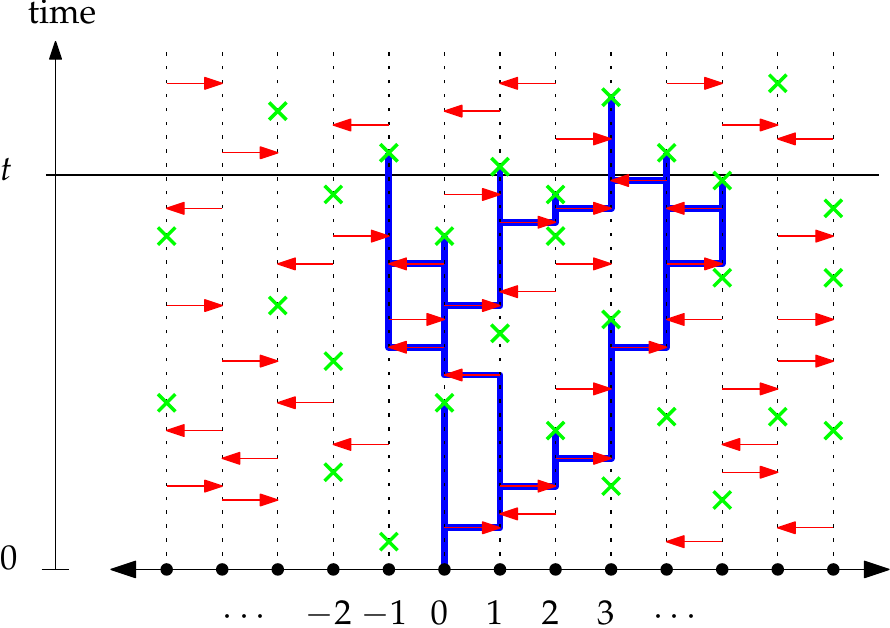}
\caption{\label{fig:contactgeometric}Graphical construction of the contact process on $\Z$. Green crosses are recovery times, red arrows are infection times. The blue graph represents the history of the contact process started with site $0$ infected: we can see that at the time $t$ only $-1$, $1$, $3$ and $4$ are infected and that it dies out in finite time.}
\end{center}
\end{figure}

We use the notation $\xi^A$ to emphasize the starting configuration $A$ of the process. We will also need to consider the contact process defined on a subset $W \subset V$ of the vertices which we will denote by $\xi_{|W}$. This process is constructed using the same graphical representation and by keeping only the infection arrows linking vertices inside $W$. Finally, we define the number of infections exiting $W$ up to time $T$ as the number of oriented arrows in the graphical representation of the form $(x,t)\to (y,t)$  with $x\in W$, $y\notin W$ and $t<T$ that are reached by an active path starting from a vertex in $\xi(0)$ and which stays inside $W$.

A nice property of the graphical construction is that it provides a natural coupling between processes defined in distinct sub-graphs and/or with distinct initial infected sets. More precisely, for any $W\subset W'$ and any $A\subset A'$, we have
\begin{equation}\label{couplingCP}
\xi^A_{|W}(t) \subset \xi^{A'}_{|W'}(t)\quad\hbox{for all $t\geq0$.}
\end{equation}
If $(W_n)$ is an increasing sequence of finite sub-graphs of $V$ such that $\lim\uparrow W_n = V$, it follows from this coupling  and \eqref{defCP} that
\begin{equation*}
\xi(t) = \lim_{n\to\infty}\uparrow \xi_{|W_n}(t) \quad\hbox{for all $t\geq 0$.}
\end{equation*}
This means that the process $\xi$ defined on the infinite graph corresponds to the weak limit of the contact process defined on any increasing sequence of finite sub-graphs $W_n$. In fact, the main theorem of this section is stated in terms of $\xi$ on the infinite graph but we could restate it in terms of the contact processes $\xi_{\vert W_n}$ restricted to finite subsets which are, ultimately, the only processes we will consider during the proof.

The graphical construction of the contact process gives a direct proof of the important self-duality property of the model: since the Poisson processes $(N_x)$ and $(N_{x,y})$ are invariant by time reversal, it follows that, for any fixed time $t\geq 0$ and any two sets $A,B$, we have
\begin{equation}\label{selfduality}
\P\{ \xi^A(t)\cap B \neq \emptyset\}  = \P\{ \xi^B(t)\cap A \neq\emptyset\}.
\end{equation}

\bigskip

We can now state the main theorem of this section which provides a sufficient condition on the geometry of a graph $G$ to ensure the existence of a sub-critical phase for the contact process.

\begin{theo}\label{TheoCMPtoCP} Let $G = (V,E)$ be a locally finite connected graph. Consider the degree-weighted CMP on $G$ \emph{i.e.} with weights given by
\begin{equation}\label{hypotheo}
\quad r(x) \defeq \hbox{deg}(x) \ind_{\{\hbox{deg}(x) \geq \Delta\}}.
\end{equation}
Suppose that for some expansion exponent $\alpha\geq\frac{5}{2}$ and some $\Delta\geq 0$, the partition $\cmp(r,\alpha)$ has no infinite cluster. Then, the contact process on $G$ has a sub-critical phase: there exists $\lambda_0 > 0$ such that, for any infection parameter $\lambda<\lambda_0$, the process starting from a finite configuration of infected sites dies out almost surely.
\end{theo}

\begin{rema} Let us make a few comments on Theorem \ref{TheoCMPtoCP}:
\begin{itemize}
\item First, we find it remarkable that, in a way, all the geometry of the graph needed to prove the existence of a sub-critical phase is encoded in the merging procedure: the radii $r(x)$ give the degrees sites but provide no information on the local shape or growth of the graph around a site. In particular, the theorem requires no assumption on the growth rate of $G$.

\item The exponent $5/2$ is not optimal. However, the proof we describe cannot yield an exponent smaller than $2$ so we did not find it worth the effort to clutter the proof with additional technical details for very little gain. In order to get an exponent close to $1$, we believe that one needs a better understanding of the inner structure of clusters. The real challenge is to prove (or disprove) the theorem for $\alpha = 1$.
\end{itemize}
\end{rema}

Theorem \ref{theo:CPonRGGandDelaunay} stated in the introduction of the paper is now a consequence of the Theorem above and Proposition \ref{TheoCMPtoCP}:
\begin{coro} The contact process on a random geometric graph or on a Delaunay triangulation admits a sub-critical phase.
\end{coro}

Let us give a rough description of the proof of Theorem \ref{TheoCMPtoCP}. The basic idea behind the theorem is that when the infection parameter $\lambda$ is very small, the sets where the contact process is locally super-critical are the big clusters of the CMP. Yet, when the CMP has no infinite cluster, these big clusters look like islands surrounded by an ocean of small degree sites; and, on this ocean, the contact process is sub-critical and dies out quickly. We will prove that we can find a neighbourhood $S$ around each cluster $C$ that will compensate the super-critical activity inside the cluster. More precisely, we will show that when an infection reaches the big cluster $C$, even though it will likely generate many infections before the whole cluster recovers, only very few infections will exit the neighbourhood $S$ (less than one in average). Then, we can couple our process with a sub-critical branching Markov chain to conclude that the process dies out. The difficulty here is that we need these estimates to hold for every single cluster. Otherwise the coupling is useless since the branching process can survive locally on the finite graphs where the estimates fail.

A natural candidate for the neighbourhood $S$ is the stabiliser of the cluster. It turns out that we need to consider a slight modification of these sets in order to have more control on their size compared to the size of the cluster $C$. This is the purpose of Section \ref{sec:etastabilisers} where we define the notion of $\eta$-stabilisers. We also prove in this section the key  Proposition  \ref{prop_ineqconv} which tells us that, indeed, these $\eta$-stabilisers are large enough to dissipate most of the infections generated by their cluster. This is where we require $\alpha \geq 2.5$ in order to have enough room to bootstrap the result from smaller $\eta$-stabilisers to larger ones. Again, the proof makes heavy use of the multi-scale structure of the CMP since stabilisers are themselves composed of smaller stabilisers. 

In Section \ref{sec:TDBRW} we introduce the particular branching process that we will couple with the contact process and present estimates for the extinction time and number of particles created that will be needed for the last steps of the proof. 

Finally, in Section \ref{sec:prooftheoCP}, we put everything together, prove the main estimates and complete the proof of the theorem.

\subsection{\texorpdfstring{$\eta$}{eta}-stabilisers and the graph \texorpdfstring{$G^\eta$}{G eta}}\label{sec:etastabilisers}

Recall that, according to Corollary  \ref{coro:stabdescend}, the stabiliser $\kS_U$ of a subset $U$ is equal to the union of all the clusters intersecting $U$ together with all their descendants in the oriented graph on the set of clusters $\cmp$. Fix $0 < \eta \leq 1$ and consider another adjacency relation on the set of clusters $\cmp$ given by
$$
C \overset{\eta}{\to} C' \quad \Longleftrightarrow \quad C\neq C' \hbox{ and } d(C,C') \leq \eta (r(C))^{\alpha}.
$$
For $\eta = 1$, this corresponds to the previous definition and for $\eta <1$, it is a more restrictive condition so the oriented graph $(\cmp,\cdot \overset{\eta}{\to} \cdot )$ is a sub-graph of $(\cmp,\cdot \to \cdot )$ defined in Section \ref{sec:orderoriented}. In particular, it does not contain any cycle or any infinite oriented path. Mimicking the definition of stabilisers, we introduce:
\begin{defi}
For any subset $W \subset V$, we call $\eta$-stabiliser of $W$ the union of all the clusters of $\cmp$ that intersect $W$ together with all their descendants in the oriented graph $(\cmp,\cdot \overset{\eta}{\to} \cdot )$. We denote this set by $\kS^{\eta}_W$, and write $\kS^{\eta}_{x}$ for $\kS^{\eta}_{\{x\}} = \kS^{\eta}_{\cmp_x}$

\end{defi}

\begin{rema}
Contrarily to $1$-stabilisers, $\eta$-stabilisers are not necessarily connected sets. See Figure \ref{fig:etastab} for an example.
However, by construction, we still have the property that any two $\eta$-stabilisers are either disjoint or one of them is included in the other one.
\end{rema}

\begin{figure}[!t]
\begin{center}
\includegraphics[width=13cm]{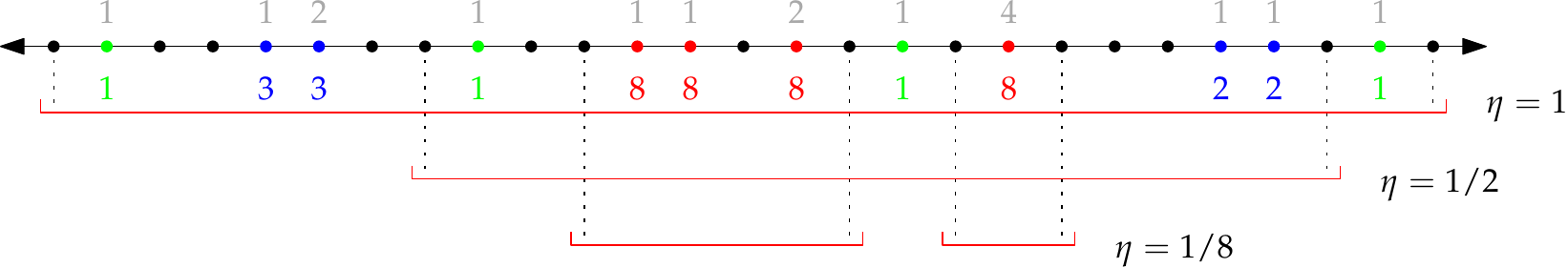}
\caption{\label{fig:etastab}Examples of $\eta$-stabilisers. The underlying graph is the same as in Figure \ref{fig:stabilisators}. Initial weights are in grey and weights of clusters are colored. Some $\eta$-stabilisers of the red cluster of total weight $8$ are displayed in red brackets. Notice that the $1/8$-stabiliser is not a connected graph.}
\end{center}
\end{figure}

We use $\eta$-stabilisers to define yet another new oriented graph, $G^\eta = (V,\cdot\overset{\kS^\eta}{\to}\cdot)$, with the same vertex set $V$ as the original graph $G$ but with adjacency relation $\cdot\overset{\kS^\eta}{\to}\cdot$ given by
\begin{equation}
x \overset{\kS^\eta}{\to} y \quad \Longleftrightarrow \quad d(\kS^\eta_x,y) = 1
\end{equation}
\emph{i.e.} the descendants of a site $x$ in this new oriented graph are exactly the sites on the outer boundary of its $\eta$-stabiliser. Notice that if $\cmp$ has no infinite cluster, then every stabiliser is finite so the out-degrees in $G^\eta$ are finite.

\begin{rema} The relation $\overset{\kS^\eta}{\to}$ is defined on the vertex set $V$ whereas $\to$ and $\overset{\eta}{\to}$ are defined on the set of clusters $\cmp$.
\end{rema}
One of the main advantages of dealing with $\eta$-stabilisers instead of $1$-stabilisers is that we have a precise control of their size which, in turns, provides sharp estimates for the distance between two adjacent sites in  $G^\eta$.

\begin{prop}\label{propsizestab} Let $x,y\in V$ such that $x\overset{\kS^\eta}{\to}{y}$, we have
$$
\eta r(\cmp_x)^\alpha \leq d(x,y) \leq 1 + \gamma r(\cmp_x)^\alpha
$$
where
\begin{equation}\label{defgamma}
\gamma  \defeq \frac{1}{2^\alpha - 2} + \frac{\eta}{1-\eta}\left(1 + \frac{1}{2^\alpha - 2}\right).
\end{equation}
\end{prop}
\begin{rema}
For $\eta = 1$, the proposition fails and the upper bound is of order  $r(\cmp_x)^{\alpha + 1}$.
\end{rema}
\begin{proof} By definition, any site $z\in V$ with $d(\cmp_x,z) \leq \eta r(\cmp_x)^\alpha$ belongs to the $\eta$-stabiliser $\kS^\eta_x$. This proves the lower bound. For the upper bound, fix $z\in \kS^\eta_{x}$. There exists a chain of clusters $\cmp_x = C_0 \overset{\eta}{\to} C_1 \overset{\eta}{\to} C_2 \overset{\eta}{\to} \ldots \overset{\eta}{\to} C_n  = \cmp_z$ for some $n\geq 0$. By definition, we have $d(C_i,C_{i+1}) \leq \eta r(C_i)^\alpha$. This implies that $r(C_{i+1})^\alpha < \eta r(C_{i})^\alpha$ since otherwise $C_i$ and $C_{i+1}$ would have merged together during the CMP. Thus, for all $0\leq i \leq n$,
$$
r(C_{i})^\alpha  \leq \eta^i r(\cmp_x)^\alpha.
$$
In view of Proposition \ref{prop_diam_weight}, we find that
$$
\diam(C_{i}) \leq \frac{\eta^i}{(2^\alpha - 2)}r(\cmp_x)^\alpha
$$
and the triangle equality yields
\begin{align*}
d(x,z) &\leq \diam(\cmp_{x}) + \sum_{i=1}^{n} \left(d(C_{i-1},C_{i}) + \diam(C_{i})\right)\\
&\leq  \frac{1}{2^\alpha - 2} r(\cmp_x)^\alpha +  \sum_{i=1}^{n} \left(\eta^i r(\cmp_x)^\alpha + \frac{\eta^i}{2^\alpha - 2}r(\cmp_x)^\alpha\right) \\
&=  \left(\frac{1}{2^\alpha - 2} + \frac{\eta}{1-\eta}\left(1 + \frac{1}{2^\alpha - 2}\right)\right) r(\cmp_x)^\alpha
\end{align*}
which yields the upper bound.
\end{proof}

\begin{prop}\label{prop_ineqconv}
Assume that $\alpha = 2.5$ and $\eta = 0.1$. Define
\begin{equation}\label{deftilder}
\tilde{r}(C) \defeq  r(C) + 2 \quad \hbox{for any cluster $C\in\cmp$.}
\end{equation}
Fix $C\in \cmp$ and consider a chain
$$
x_0 \overset{\kS^\eta}{\to} x_1 \overset{\kS^\eta}{\to} \ldots \overset{\kS^\eta}{\to} x_n \quad \hbox{where $x_i \in \kS^\eta_C \setminus C$ for all $i=0,1,\ldots, n-1$}
$$
(we do not impose any restriction on $x_n$ which may either be inside $\kS_C^\eta \setminus C$ or on its outer boundary). We have, for any $\beta \in [1, \alpha]$,
\begin{equation}\label{ineqconv_weak}
\sum_{i=0}^{n-1}  \tilde{r}(\cmp_{x_i})^{\beta} \geq d(x_0,x_n)^{\frac{\beta}{\alpha}}.
\end{equation}
Moreover assuming that $d(x_0,C) = 1$, $d(x_n,\kS^\eta_C)=1$ and $r(C) \geq 100$, we have the stronger inequality, valid for $1\leq \beta \leq1.01$,
\begin{equation}\label{ineqconv_strong}
\sum_{i=0}^{n-1} \tilde{r}(\cmp_{x_i})^{\beta} \geq \beta \tilde{r}(C)^{\beta}.
\end{equation}
\end{prop}

\begin{figure}[!t]
\begin{center}
\includegraphics[width=11cm]{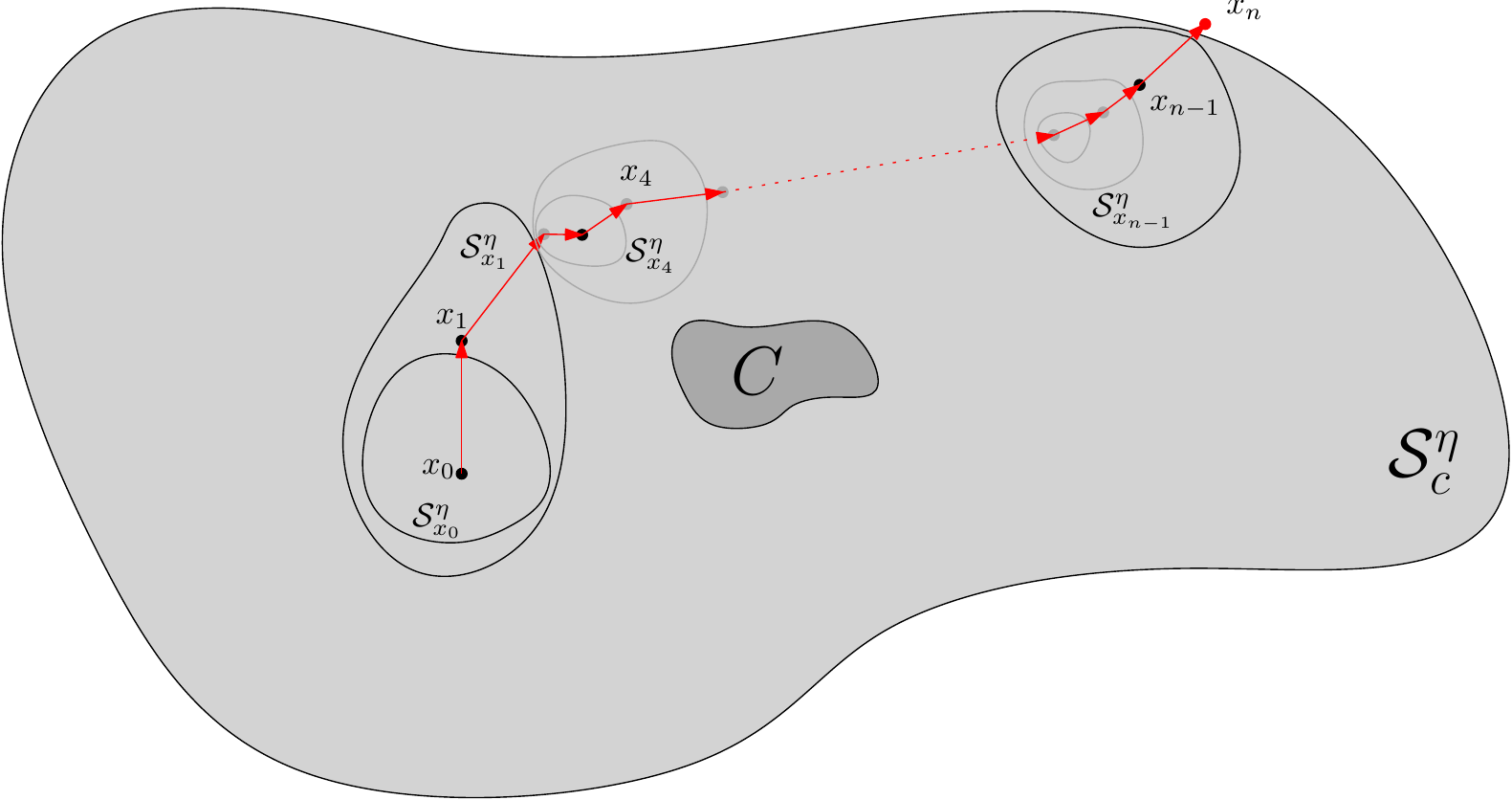}
\caption{\label{fig:eta} A cluster $C$ and its $\eta$-stabilizer $\kS^\eta_C$ with a path $(x_0,\ldots, x_n)$ in $G^\eta$ as in Proposition \ref{prop_ineqconv}. In this example, the endpoint $x_n$ in on the outer boundary of $\kS^\eta_C$.}
\end{center}
\end{figure}

Let us give an interpretation for this proposition. Imagine that, for each vertex $x\in V$, we must pay a price $\tilde{r}(\cmp_x)^{\beta}$ in order to travel along one of its outgoing edges in $G^\eta$. Inequality \eqref{ineqconv_weak} tells us that, in order to travel from some site $x$ to some other site $y$, we must pay a total price at least $d(x,y)^{\beta/\alpha}$. The second inequality \eqref{ineqconv_strong} says that when $\beta$ is close to $1$ and when the cluster $C$ is big enough, if we want to exit $\kS^\eta_C$ starting from a some boundary point of $C$, then any road staying inside $\kS^\eta_C \setminus C$ will cost more than the price $\tilde{r}(C)^{\beta}$ required to travel directly from a site of $C$ to the outer boundary of $\kS^\eta_C$.

The proof of the proposition is based on the concavity formula stated in the next lemma.
\begin{lemm}\label{lemm_conv}
Let $\varepsilon,\theta \in [0,1]$. We have
\begin{equation*}
\inf \left\{
(x_1^\theta +\ldots + x_k^\theta) \,:\,
k \in \N, \,
x_1 + \ldots + x_k = 1, \,
0 \leq x_i \leq \varepsilon
\right\}
= g(\varepsilon,\theta)
\end{equation*}
where $g(\varepsilon,\theta) \defeq \lfloor \frac{1}{\varepsilon}\rfloor \varepsilon^\theta + (1-\lfloor \frac{1}{\varepsilon}\rfloor\varepsilon)^\theta \geq 1$.
\end{lemm}
\begin{proof}[Proof of Lemma \ref{lemm_conv}.]
Set $n = \lfloor \frac{1}{\varepsilon}\rfloor$. Choosing $k = n+1$,
$x_1 = \ldots = x_n = \varepsilon$ and $x_{n+1} = 1- n\varepsilon$ we see that the infimum is indeed smaller that $n\varepsilon^\theta + (1-n\varepsilon)^\theta = g(\varepsilon,\theta)$. For the converse inequality, we use the fact that, for $a\geq b$, we have $(a+x)^\theta + (b-x)^\theta \geq a^\theta + b^\theta$ for any $x\geq 0$ and work by induction to transfer mass $\varepsilon$ on the first $n$ values $x_1,\ldots, x_n$ and the remaining mass $1-n\varepsilon$ on $x_{n+1}$. This argument also shows that $g(\varepsilon,\theta) \geq 1$ for any choice $\varepsilon,\theta \in [0,1]$. The details are left out to the reader.
\end{proof}

\begin{proof}[Proof of Proposition \ref{prop_ineqconv}.]
Set
\begin{equation*}
D = \sum_{i=0}^{n-1} \tilde{r}(\cmp_{x_i})^{\alpha} \quad \hbox{and} \quad \varepsilon  = \max_{0\leq i \leq n-1}\frac{\tilde{r}(\cmp_{x_i})^{\alpha}}{D}.
\end{equation*}
On the one hand, using Lemma \ref{lemm_conv}, we find that
\begin{equation*}
\sum_{i=0}^{n-1} \tilde{r}(\cmp_{x_i})^{\beta} =  \sum_{i=0}^{n-1} \left(\tilde{r}(\cmp_{x_i})^{\alpha}\right)^{\frac{\beta}{\alpha}} \geq g\left(\varepsilon,\frac{\beta}{\alpha}\right) D^{\frac{\beta}{\alpha}}.
\end{equation*}
On the other hand, Proposition \ref{propsizestab} states that $d(x_i,x_{i+1}) \leq 1+\gamma r(\cmp_{x_i})^\alpha \leq \gamma \tilde{r}(\cmp_{x_i})^\alpha$ hence
\begin{equation}\label{ineqtec1}
D \geq \frac{1}{\gamma} \sum_{i=0}^{n-1} d(x_i,x_{i+1}) \geq \frac{1}{\gamma}d(x_0,x_n).
\end{equation}
Combining these inequalities, we find that
\begin{equation*}
\sum_{i=0}^{n-1} \tilde{r}(\cmp_{x_i})^{\beta} \geq g\left(\varepsilon,\frac{\beta}{\alpha}\right)\left(\frac{1}{\gamma}\right)^{\frac{\beta}{\alpha}}  d(x_0,x_{n})^{\frac{\beta}{\alpha}}.
\end{equation*}
Inequality \eqref{ineqconv_weak} now follows from the fact that $g(\varepsilon,\beta/\alpha) \geq 1$ and $\gamma \simeq 0.41 < 1$.

We now turn our attention to the second inequality. Under the assumption that $d(x_0,C) = 1$ and  $d(x_n,\kS^\eta_C)=1$, the lower bound of Proposition \ref{propsizestab} gives $d(x_0,x_n) \geq \eta r(C)^\alpha - 1$ hence
\begin{eqnarray}
\nonumber\sum_{i=0}^{n-1} \tilde{r}(\cmp_{x_i})^{\beta} &\geq& g\left(\varepsilon,\frac{\beta}{\alpha}\right)\left(\frac{1}{\gamma}\right)^{\frac{\beta}{\alpha}} \left(\eta r(C)^\alpha - 1\right)^{\frac{\beta}{\alpha}}\\
\label{ineqtec2} &=&
g\left(\varepsilon,\frac{\beta}{\alpha}\right)\left(\frac{\eta}{\gamma}  - \frac{1}{\gamma r(C)^\alpha}\right)^{\frac{\beta}{\alpha}}r(C)^{\beta}.
\end{eqnarray}
Since $x_0,\ldots,x_{n-1}$ are all in $\kS^\eta_C \setminus C$, we have $r(\cmp_{x_i})^\alpha \leq \eta r(C)^\alpha$ and therefore
\begin{equation*}
\tilde{r}(\cmp_{x_i})^\alpha \leq \left( 2 + \eta^{\frac{1}{\alpha}} r(C)\right)^\alpha
\end{equation*}
but, according to \eqref{ineqtec1},
\begin{equation*}
D \geq \frac{1}{\gamma}d(x_0,x_n) \geq \frac{1}{\gamma} (\eta r(C)^\alpha - 1).
\end{equation*}
These last two inequalities combined together yield
\begin{equation*}
\varepsilon \leq \frac{\left( 2 + \eta^{\frac{1}{\alpha}} r(C)\right)^\alpha}{\frac{1}{\gamma} (\eta r(C)^\alpha - 1)} = \gamma\frac{\left(\frac{2}{\eta^{\frac{1}{\alpha}}r(C)} + 1 \right)^\alpha}{\left(1 - \frac{1}{\eta r(C)^\alpha}\right)}
\end{equation*}
For our particular choice of the parameters $\alpha = 2.5$, $\eta = 0.1$ and when $r(C) \geq 100$, one can check that $\varepsilon \leq 0.47$ which, in turn, implies that for $1\leq \beta \leq 1.01$
\begin{equation*}
g\left(\varepsilon,\frac{\beta}{\alpha}\right)\left(\frac{\eta}{\gamma}  - \frac{1}{\gamma r(C)^\alpha}\right)^{\frac{\beta}{\alpha}} \geq 1.01 \geq \beta.
\end{equation*}
This inequality together with \eqref{ineqtec2} completes the proof of the proposition.
\end{proof}

\subsection{A time-delayed branching Markov chain}\label{sec:TDBRW}

We now introduce another stochastic process that we will use to dominate the contact process.
This process will be defined on the oriented graph $G^\eta$ but we give here a general description. Let $H = (W,\cdot\to\cdot)$ denote an oriented graph. We assume that every site has finite out-degree. For each $x\in W$, consider a real-valued random variable
\[
\tau_x > 0
\]
together with a family of integer-valued random variables
\[
B_{x,y_1} \, ; \, \ldots \, ; \, B_{x,y_n} \in \lin 0 ; \infty \irin
\]
where $\{y_1,\ldots y_n\}$ is the set of neighbor sites $x\to y_i$. We do not assume any independence between the random variables $(\tau_x,B_{x,y_1},\ldots,B_{x,y_n})$. We call Time-Delayed Branching Markov Chain (TDBMC) a continuous-time system of particles $(X_t(x),x\in W)$ such that:
\begin{itemize}
\item $X_t(x) \in \{0,1,\ldots\}$ represents the number of particles at site $x$, at time $t$. We start from an initial configuration of particles $(X_0(x),x\in W)$ at time $0$. Note that there may be more than one particle per site.
\item Each particle evolves independently of the others: when a particle is created at some time $t$, its survival time and progeny is independent of the evolution of all the other particles in the system at time $t$.
\item When a particle is created at some site $x$, it stays there for a random time $\hat{\tau}_x$ after which it disappears while giving birth to $\hat{B}_{x,y}$ new particles at each neighboring site $x\to y$ with
$$
(\hat{\tau}_x,\hat{B}_{x,y_1},\ldots,\hat{B}_{x,y_n}) \overset{\hbox{\scriptsize{law}}}{=} (\tau_x,B_{x,y_1},\ldots,B_{x,y_n}).
$$
\end{itemize}
In order to define such a process at every time $t\geq 0$, we must ensure that there can be no explosion in finite time. This is the case as soon as
\begin{equation}\label{mintau}
\inf_{x\in W}\E[\tau_x] > 0
\end{equation}
which will be a standing assumption from now on. Let us remark that if $\tau_x = 1$ a.s. for all $x$, then $X$ is a classical discrete time Branching Markov Chain. Another special case is when all the $\tau_x$'s have exponential distribution; then the TDBMC is a continuous time Markov process. Notice however that despite its name, the process $X$ does not in general satisfy the Markov property.

We use the notation $\E_{x}[\cdot]$ to denote the expectation for the process started at time $0$ from one single particle located at site $x$. Define also
\begin{equation}\label{notaTDBMC}
b_x \defeq \E\Big[\sum_{x\to y} B_{x,y}\Big],\quad \lambda_{x,y} \defeq \frac{\E[B_{x,y}]}{b_x} \quad\hbox{ and }\quad u_x \defeq \E[\tau_x],
\end{equation}
with the convention $\lambda_{x,y}=0$ if $b_x = 0$. The next proposition collects properties of this process that we will use.

\begin{prop}\label{propTDBMC} Let $x_0\in W$. Consider the TDBMC $X$ started from a single particle located at site $x_0$.
\begin{enumerate}
\item \label{TDBMC-1} Let $N$ be the total number of particles born in the TDBMC up to time $+\infty$, we have
\[
\E_{x_0}[N] \leq  \sum_{n=0}^\infty \,\sup_{(x_0,\ldots,x_n)\in \mathcal{P}^n_{x_0}}\left(
\prod_{i=0}^{n-1} b_{x_i}\right)
\]
where $\mathcal{P}^n_{x_0}$ is the set of all oriented paths of length $n$ starting at $x_0$.  By convention, the product over an empty index set is equal to $1$.
\item \label{TDBMC-2} Let $T \in [0,+\infty]$ denote the extinction time of $X$. We have
\[
\E_{x_0}[T] \leq \sum_{n=0}^\infty \, \sup_{(x_0,\ldots,x_n)\in \mathcal{P}^n_{x_0}}\left(
u_{x_n}\prod_{i=0}^{n-1} b_{x_i} \right).
\]
\item \label{TDBMC-3} Let $U \subset W$ with $x_0 \notin U$. Consider a modification of the process where all the particles entering $U$ are frozen (\emph{i.e.} when a particle reaches $U$, it does not reproduce and it stays there forever). Then, starting from one particle located at $x_0$, we have
\begin{equation*}
\E_{x_0}\left[
\begin{gathered}
\text{\small{Total number of particles frozen}}\\
\text{\small{in $U$ up to time $t=+\infty$}}
\end{gathered}
\right] \leq
\sup_{(x_0,\ldots,x_n)\in \mathcal{P}^U_{x_0}}\left( \prod_{i=0}^{n-1}b_{x_i}\right)
\end{equation*}
where $\mathcal{P}^U_{x_0}$ is the set of all finite oriented paths $(x_0,\ldots,x_n)$ starting at $x_0$, with $x_i \in W \setminus U$ for $i<n$, and with $x_n \in U$.
\end{enumerate}
\end{prop}

\begin{proof}
Starting from one particle at site $x_0$ and conditioning on its progeny, we get the relation
\begin{equation*}
\E_{x_0}[N] = 1 + \sum_{x_0\to x_1}\E[B_{x_0,x_1}]\E_{x_1}[N] = 1 + \sum_{x_0\to x_1}b_{x_0}\lambda_{x_0,x_1}\E_{x_1}[N].
\end{equation*}
Expanding this induction relation we get by monotone convergence
\begin{align*}
\E_{x_0}[N] & = \sum_{n=0}^{\infty} \, \sum_{(x_0,\ldots,x_n) \in \mathcal{P}^n_{x_0}} \, \prod_{i=0}^{n-1}b_{x_i}\lambda_{x_{i},x_{i+1}}\\
&\leq  \sum_{n=0}^{\infty} \left(\sup_{(x_0,\ldots,x_{n}) \in \mathcal{P}^n_{x_0}} \, \prod_{i=0}^{n-1}b_{x_{i}}\right)
\left(\sum_{(x_0,\ldots,x_n) \in \mathcal{P}^n_{x_0}} \, \prod_{i=0}^{n-1}\lambda_{x_{i},x_{i+1}}\right).
\end{align*}
Recalling that $(\lambda_{x,y})$ is a (possibly defective) transition kernel \emph{i.e.} $\sum_{y\sim x}\lambda_{x,y} \leq 1$, the sum after the supremum on the right hand side of the last inequality is bounded above by $1$ which completes the proof of the first statement.

The proof of Item \ref{TDBMC-2} is obtained similarly starting from the inequality
\begin{equation*}
\E_{x_0}[T] \leq \E[\tau_{x_0}] + \sum_{x_0\to x_1}\E[B_{x_0,x_1}]\E_{x_1}[T] = u_{x_0} + \sum_{x_0\to x_1}b_{x_0}\lambda_{x_0,x_1}\E_{x_1}[T]
\end{equation*}
which gives, using the same bounds,
\begin{equation*}
\E_{x_0}[T] \leq \sum_{n=0}^{\infty} \, \sum_{(x_0,\ldots,x_n) \in \mathcal{P}^n_{x_0}} u_{x_n} \, \prod_{i=0}^{n-1}b_{x_i}\lambda_{x_{i},x_{i+1}}
\leq \,  \sum_{n=0}^{\infty} \sup_{(x_0,\ldots,x_{n}) \in \mathcal{P}^n_{x_0}}\left(u_{x_n}\prod_{i=0}^{n-1}b_{x_{i}}\right).
\end{equation*}

We now prove Item \ref{TDBMC-3}. The number of particles frozen does not depend on the $\tau_x$'s, so we just need to consider the case $\tau_x = 1$ a.s. for all $x\in W$. But in this case, the TDBMC is simply a discrete-time branching Markov chain. Let $X'_{n}(x)$ denote the number of particles at time $n$ and site $x$ for the discrete-time branching Markov chain obtained by freezing particles in $U$. Its transition kernel $(B'_{x,y})_{x,y\in W}$ is given by
\[
B'_{x,y} =
\begin{cases}
\ind_{\{x=y\} }& \text{if $x\in U$,}\\
B_{x,y}& \text{otherwise}
\end{cases}
\]
(here, we implicitly added a loop at each site $x\in U$ so that $x\to x$). Now, define
\[
b'_x =
\begin{cases}
1& \text{if $x\in U$,}\\
b_x& \text{otherwise.}
\end{cases}
\quad \text{and} \quad
\lambda'_{x,y} \defeq
\begin{cases}
\E[B'_{x,y}]/b'_x& \text{if $b'_x >0$,}\\
0& \text{otherwise.}
\end{cases}
\]
Once again, we have a recurrence relation, namely
\begin{equation*}
\E_{x_0}[X'_n(x)] = \sum_{y\to x} \E_{x_0}[X'_{n-1}(y)]\lambda'_{y,x}b'_y
\end{equation*}
which implies
\[
\E_{x_0}[X'_n(x)] = \sum_{
\substack{
(x_0,\ldots ,x_{n})\in \mathcal{P}^n_{x_0}\\
x_n = x
}
}
\, \prod_{i=0}^{n-1} b'_{x_i}\lambda'_{x_i,x_{i+1}}.
\]
(the paths in $\mathcal{P}^n_{x_0}$ above are considered for the new graph where there is a loop at each site of $U$). Summing over all $x\in U$, we find that
\begin{align*}
\E_{x_0}\left[
\begin{gathered}
\text{\small{Total number of particles}}\\
\text{\small{frozen in $U$ up to time $n$}}
\end{gathered}
\right]
& = \sum_{\substack{
(x_0,\ldots ,x_{n})\in \mathcal{P}^n_{x_0}\\
x_n \in U}
}
\, \prod_{i=0}^{n-1} b'_{x_i}\lambda'_{x_i,x_{i+1}}\\
&\leq \left( \sup_{
\substack{
(x_0,\ldots,x_k)\in \mathcal{P}^U_{x_0}\\
k \leq n
}}
\, \prod_{i=0}^{k-1}
b_{x_i} \right)
\left(
\sum_{
\substack{
(x_0,\ldots ,x_{n})\in \mathcal{P}^n_{x_0}\\
x_n \in U
}
}
\, \prod_{i=0}^{n-1} \lambda'_{x_i,x_{i+1}}
\right)\\
&\leq \sup_{
\substack{
(x_0,\ldots,x_k)\in \mathcal{P}^U_{x_0}\\
k \leq n
}}
\, \prod_{i=0}^{k-1}
b_{x_i}.
\end{align*}
We conclude the proof by letting $n$ go to infinity.
\end{proof}

\subsection{Proof of Theorem \ref{TheoCMPtoCP}}\label{sec:prooftheoCP}

We can now state the main estimates which assert that, starting from
a completely infected cluster $C$, the expected number of infections exiting the $\eta$-stabiliser $\kS^\eta_C$ for the contact process inside $\kS^\eta_C$ decreases faster than exponentially with respect to the cluster's weight.

\begin{prop}[main estimates]\label{prop-mainestimates} Fix $\eta=0.1$ and $\alpha = 2.5$. Suppose that the CMP with parameters given by \eqref{hypotheo} has no infinite cluster for some $\Delta\geq 0$. Then, there exists  $\lambda_0 >0$ depending only on $\Delta$ such that, for any cluster
$C\in \cmp$ and any infection rate $\lambda \leq \lambda_0$, we have
\begin{equation}\label{main-exit}
\E\left[
\begin{gathered}
\text{\small{Total number of infections exiting $\kS_C^\eta$}}\\
\text{\small{for the contact process $\xi^C_{|\kS^\eta_C}$}}
\end{gathered}
\right]
\leq \frac{1}{2}e^{-\tilde{r}(C)^{1.01}}
\end{equation}
(where $\tilde{r}(C) = r(C) + 2$ as in \eqref{deftilder}) and
\begin{equation}\label{main-time}
\E\left[
\text{\small{Extinction time of the contact process $\xi^C_{|\kS^\eta_C}$}}\\
\right]
\leq e^{3 r(C)}.
\end{equation}
\end{prop}

The strength of these estimates is that $\lambda_0$ only depends on the geometry of $G$ through $\alpha$ and $\Delta$. Thus, the proposition gives bounds that are uniform for any cluster of any graph whose associated CMP has no infinite cluster.  Let us first show how these estimates easily imply Theorem \ref{TheoCMPtoCP}.

\begin{proof}[Proof of Theorem \ref{TheoCMPtoCP}.] By monotonicity of the CMP w.r.t. the expansion exponent $\alpha$, we only need to prove the theorem for $\alpha = 2.5$. Recall that $G^\eta = (V,\cdot\overset{\kS^\eta}{\to}\cdot)$ denotes the oriented graph with vertex set $V$ where the sites $y$ such that $x\overset{\kS^\eta}{\to}y$ are exactly those on the outer boundary of the $\eta$-stabiliser of $x$. Consider the following modification of the contact process where there may be more than one infection at each site.
\begin{itemize}
\item We start from an initial infected vertex $x_0$. 
\item At time $0$, we instantaneously infect every site of the cluster $\cmp_{x_0}$. Then, we run a contact process inside $\kS_{x_0}^{\eta}$, freezing every infection exiting the $\eta$-stabiliser.  
\item At the time when the contact process inside $\kS_{x_0}^{\eta}$ dies out, for each infection that exited $\kS_{x_0}^{\eta}$, we restart an independent process from the endpoint of the infection. This means that, for each infection with endpoint, say $z$, we instantaneously infect every site of $\cmp_z$ and then run an independent contact process inside $\kS_{z}^{\eta}$, freezing the all the infections exiting the $\eta$-stabiliser.
\item We construct the process for all times (or until extinction) by iterating this procedure.  
\end{itemize}
The freezing of particles in this modified process prevents us from coupling it at deterministic times with the real contact process started from the same initial infected site $x_0$. Yet, we can still construct both processes on the same probability space in such way that:
\begin{enumerate}
\item \label{CouplingExtinction} The extinction time of the modified process is larger than or equal to the extinction time of the contact process.
\item For any directed edge of the graph $G$, the total number of infections sent through the edge by the contact process is smaller than or equal to the number of infections sent by the modified process.
\end{enumerate}
This coupling can easily be achieved by using the same sequences of clocks on sites and oriented arrows for both processes. However, contrarily to the graphical construction, in this case, the time on a site (resp. oriented edge) runs only when the site (resp. start vertex) is infected. This ensures that the modified process will never miss any infection clock that the contact process uses. We leave the details to the reader. 

\bigskip

Looking at infections exiting $\eta$-stabilisers, we see that this modified process naturally defines a TDBMC $X$ on the graph $G^\eta$ with transition kernel given by (with the notation of Section \ref{sec:TDBRW}):
\begin{equation}\label{deftau}
\tau_x \defeq \text{Extinction time of the contact process $\xi^{\cmp_x}_{|\kS^\eta_x}$ }
\end{equation}
and for every $y\in V$ such that $x\overset{\kS^\eta}{\to}y$,
\begin{equation}\label{defB}
B_{x,y} \defeq
\begin{cases}
\text{total number of infections reaching} \\
\text{site $y$ for the contact process $\xi^{\cmp_x}_{|\kS^\eta_x}$.}
\end{cases}
\end{equation}
From the coupling with the contact process, we see that condition \ref{CouplingExtinction} on extinction times means that
\begin{equation*}
\inf(t\geq 0, \xi^{x_0}(t) = \emptyset) \leq \inf(t\geq 0, X_t(x) = 0\hbox{ for all $x\in V$}).
\end{equation*}
Using the notation of \eqref{notaTDBMC}, the main estimates translate to
\begin{equation}\label{Zest}
b_x \leq \frac{1}{2}e^{-\tilde{r}({\cmp_x})^{1.01}} \quad \hbox{ and } u_x \leq e^{3 r({\cmp_x})}.
\end{equation}
Therefore, in view of Item \ref{TDBMC-1} of Proposition \ref{propTDBMC}, we find that the expected total number of particles created in $Z$ starting from $x_0$ is bounded by
$$
\sum_{n=0}^\infty \sup_{(x_0,\ldots,x_n)\in \mathcal{P}^n_{x_0}} \prod_{i=0}^{n-1} b_{x_i} \leq \sum_{n=0}^\infty \frac{1}{2^n}.
$$
(here, $\mathcal{P}^n_{x_0}$ denotes the set of paths in $G^\eta$ starting from $x_0$ with length $n$). This means that $X$ creates only finitely many particles a.s., hence its extinction time $T$ is also finite. This in turn implies that the contact process dies out almost surely.
\end{proof}

\begin{proof}[Proof of Proposition \ref{prop-mainestimates}]

The proof works by induction on the weights $r(C)$ of clusters. From now on, we fix
\begin{equation*}
\alpha = 2.5, \quad \Delta \geq 1, \quad \eta = 0.1
\end{equation*}
and we assume that the CMP defined by \eqref{hypotheo} has no infinite cluster. Let us remark that, since, for every cluster $C$, the stabiliser $\kS^\eta_C$ is a finite set, the random variables inside the expectations in \eqref{main-exit} and \eqref{main-time} have exponential tails and are therefore finite. This follows easily from the graphical construction described previously.

For every $R > 1$, there exist only finitely many graphs isomorphic to some $\eta$-stabiliser $\kS^\eta_C$ where $C$ is a cluster with $r(C) \leq R$. This see this, notice that $r(C) \leq R$ implies that every site in the stabiliser has degree at most $R\vee\Delta$. The number of sites in the cluster $C$ is also bounded by $R$. Using Proposition \ref{propsizestab}, we deduce that the number of sites in the stabiliser $\kS^\eta_C$ is also bounded by $f(R)$ for some function $f$ growing fast enough. This proves our assertion since there are only finitely many non-isomorphic graphs with at most $f(R)$ vertices with degrees bounded by $R\vee \Delta$.

For each cluster $C$, when the infection parameter $\lambda$ of the contact process goes to $0$, the expectations in \eqref{main-exit} and \eqref{main-time} tend to $0$ by dominated convergence. This means that we only need to prove Proposition \ref{prop-mainestimates} for clusters satisfying $r(C) \geq R$ where $R = R(\Delta)$  may be chosen arbitrarily large.

From now on, let $R$ be large and $\lambda >0$  such that Proposition \ref{prop-mainestimates} holds for every cluster $C$ with $r(C) \leq R$. We will prove that the same result holds, in fact, for every cluster $C$ with $r(C) \leq \frac{1}{\eta}R$ and the result will follow by induction.

Fix $C\in \cmp$ such that $r(C) \leq \frac{1}{\eta}R$. We use following the notation for conciseness:
\begin{eqnarray*}
S &\defeq& \kS^\eta_C \\
D &\defeq& \kS^\eta_C \setminus C.
\end{eqnarray*}
We decompose the proof in six steps.

\begin{step} For any $x_0 \in D$ with $d(x_0,C)=1$, we have
\begin{equation}\label{exit1}
\E\left[
\begin{gathered}
\text{Total number of infections exiting through $\partial S$}\\
\text{for the contact process $\xi^{x_0}_{|D}$.}
\end{gathered}
\right]
\leq e^{-1.01 \tilde{r}(C)^{1.01}}
\end{equation}
(remark that we do not count in this expectation the infections going from $D$ into $C$) and
\begin{equation}\label{exittime1}
\E\left[
\text{Extinction time of the contact process $\xi^{x_0}_{|D}$ }\\
\right]
 \leq 2c
\end{equation}
where $c$ is a universal constant. 
\end{step}

We consider again the TDBMC $X$ on $G^\eta$ with transition kernel given by \eqref{deftau} and \eqref{defB}, where we freeze all the particles exiting $D$. Using the same arguments as  in the proof of Theorem \ref{TheoCMPtoCP}, we can couple $X$ started with one particle at $x_0$ with the contact process $\xi^{x_0}_{|D}$ in such a way that:
\begin{enumerate}
\item The total number of particles in $X$ frozen on the outer boundary of $S$ is larger than the total number of infections sent outside of $S$ by the contact process $\xi^{x_0}_{|D}$.
\item The extinction time of $X$ is larger than the extinction time of the contact process $\xi^{x_0}_{|D}$.
\end{enumerate}
Now, since every cluster $C'$ inside $D$ is such that $r(C') \leq \eta r(C) \leq R$, we can use the main estimate to upper bound the quantities $(b_x, x\in D)$. Denoting by $\mathcal{P}^{V \setminus S}_{x_0}$ the set of paths in $G^\eta$ which start from $x_0$ and such that $x_i \in D$ for $i<n$ and $x_n \in V \setminus S$, we find with the help of Item \ref{TDBMC-3} of Proposition \ref{propTDBMC} that
\begin{align*}
\E\left[
\begin{gathered}
\text{Total number of infections exiting through $\partial S$}\\
\text{for the contact process $\xi^{x_0}_{|D}$.}
\end{gathered}
\right]
&\leq  \sup_{(x_0,\ldots,x_n)\in \mathcal{P}^{V \setminus S}_{x_0}}\left( \prod_{i=0}^{n-1}b_{x_i}\right)\\
&\leq \sup_{(x_0,\ldots,x_n)\in \mathcal{P}^{V \setminus S}_{x_0}}\!\!\left( \frac{1}{2^n}e^{-\sum_{i=0}^{n-1}\tilde{r}(\cmp_{x_i})^{1.01}}\right)\\
&\leq e^{-1.01 \tilde{r}(C)^{1.01}}
\end{align*}
where we used Proposition \ref{prop_ineqconv} for the last inequality. This completes the proof of \eqref{exit1}. The proof of the second inequality is similar. Let $\mathcal{P}^n_{x_0}$ denote the set of paths in $G^\eta$ staying inside $D$, starting at $x_0$ and of length $n$. We use Item \ref{TDBMC-2} of Proposition \ref{propTDBMC} combined with Proposition \ref{prop_ineqconv} to get that
\begin{align*}
\E\left[
\begin{gathered}
\text{Extinction time of the}\\
\text{contact process $\xi^{x_0}_{|D}$}
\end{gathered}
\right] 
&\leq  \sum_{n=0}^\infty \sup_{(x_0,\ldots,x_n)\in \mathcal{P}^n_{x_0}}\left(
u_{x_n}\prod_{i=0}^{n-1} b_{x_i} \right)\\
&\leq \sum_{n=0}^{\infty}\sup_{(x_0,\ldots,x_n)\in \mathcal{P}^n_{x_0}}\left( \frac{e^{3 r(\cmp_{x_n})}}{2^n}e^{-\sum_{i=0}^{n-1}\tilde{r}(\cmp_{x_i})^{1.01}}\right)\\
&\leq \sum_{n=0}^{\infty}\sup_{(x_0,\ldots,x_n)\in \mathcal{P}^n_{x_0}}\left( \frac{e^{ 3(1 + d(x_0,x_n))^{\frac{1}{\alpha}}}}{2^n}e^{-d(x_0,x_n)^{\frac{1.01}{\alpha}}}\right)\\
&\leq \sum_{n=0}^{\infty} \frac{c}{2^n} \leq 2 c
\end{align*}
where $c$ is the overall supremum on $[0,\infty)$ of the function $x\to 3(1+x)^{\frac{1}{\alpha}} -x^{\frac{1.01}{\alpha}}$.

\begin{step} We have
\begin{equation}
\P\left\{
\begin{gathered}
\text{The contact process $\xi^{D}_{|D}$ never }\\
\text{sends any infection into $C$}
\end{gathered}
\right\} \geq \frac{1}{2^{r(C)}}.
\end{equation}
\end{step}

Let $e_1,\ldots,e_m$ denote the set of edges connecting $D$ to $C$. Using the fact that the contact process has positive correlation at all times (\emph{c.f.} Theorem B$17$, p$9$ of \cite{Li2}) it is easy to check  that
\begin{align*}
\P\left\{
\begin{gathered}
\text{The contact process $\xi^{D}_{|D}$ never }\\
\text{sends any infection into $C$}
\end{gathered}
\right\}
 &=
 \P\left\{
\begin{gathered}
\text{The contact process $\xi^{D}_{|D}$ never sends}\\
\text{any infection through  $e_i$ for all $i\leq m$.}
\end{gathered}
\right\}\\
&\geq
\prod_{i=1}^m
\P\left\{
\begin{gathered}
\text{The contact process $\xi^{D}_{|D}$ never }\\
\text{sends any infection through $e_i$.}
\end{gathered}
\right\}
\end{align*}
Fix $e_i = (x_i \to z_i)$ with $x_i \in D$ and $z_i \in C$. Let $I_i$ denote the total time site $x_i$ spends infected:
\begin{equation*}
I_{i} \defeq \int_{0}^\infty \ind_{\left\{ \xi_{|D}^{D}(t) \cap \{x_i\} \neq \emptyset\right\} }dt.
\end{equation*}
Using the self-duality property of the contact process we find that
\begin{align*}
\E[I_{i}] &= \int_{0}^\infty \P\left\{ \xi_{|D}^{D}(t) \cap \{x_i\} \neq \emptyset\right\} dt\\
&= \int_{0}^\infty \P\left\{ \xi_{|D}^{x_i}(t) \cap \{D\} \neq \emptyset\right\} dt\\
&= \E \left[ \int_{0}^\infty \ind_{\left\{ \xi_{|D}^{x_i}(t) \neq \emptyset\right\} }dt \right]\\
&= \E\left[
\text{Extinction time of the contact process $\xi^{x_i}_{|D}$ }
\right]\\
&\leq 2c
\end{align*}
where we used Step $1$ for the last inequality. Now, recall that an infection propagates through $e_i$ whenever site $x_i$ is infected and a clock attached to the oriented edge $x_i\to z_i$ rings. Since these clocks are independent of the contact process $\xi^{D}_{|D}$, the expected number of infections exiting $\xi^{D}_{|D}$ through $e_i$ is bounded by $2c\lambda$. We can without loss of generality assume that $2c \lambda < \frac{1}{2}$ in which case we have
\begin{equation*}
\P\left\{
\begin{gathered}
\text{The contact process $\xi^{D}_{|D}$ never}\\
\text{sends any infection through $e_i$}
\end{gathered}
\right\} \geq \frac{1}{2}.
\end{equation*}
The claim follows noticing that the number  $m$ of edges between $C$ and $D$ is bounded by $r(C)$.

\begin{step} We have
\begin{equation*}
\P\left\{
\begin{gathered}
\text{The contact process $\xi^{C}_{|C}$ never }\\
\text{sends any infection into $D$}
\end{gathered}
\right\} \geq \frac{1}{2^{r(C)}}.
\end{equation*}
and
\begin{equation}\label{extincC}
\E\left[
\text{Extinction time of $\xi^{C}_{|C}$}
\right] \leq  2^{r(C)}.
\end{equation}
\end{step}

This step is easy. We use a very crude estimate: the probability that no infection ever escapes $C$
and that the contact process dies out before time $1$ is larger than the probability that every (recovery) clock attached to a vertex of $C$ rings before time $1$ and no (infection) clock attached to an oriented edge with a start vertex in $C$ rings before time $1$. Therefore, for $\lambda < 0.1$, using again that $r(C)$ upper bounds the number of outgoing edges and vertices in $C$, we get
\begin{equation*}
\P\left\{
\begin{gathered}
\text{The contact process $\xi^{C}_{|C}$ never sends any}\\
\text{infection into $D$ and dies out before time $1$}
\end{gathered}
\right\} \geq \left((1-e^{-1})e^{-\lambda}\right)^{r(C)} \geq
\frac{1}{2^{r(C)}}.
\end{equation*}
Now, comparing the contact process $\xi^C_{C}$ with the modified process obtained by re-infecting every site of $C$ at each integer time when there is at least one infected site, it follows from the previous inequality that the extinction time of $\xi^C_{C}$ is stochastically dominated by a geometric random variable with parameter $\frac{1}{2^{r(C)}}$. This gives \eqref{extincC}.

\begin{step} Denote by $\overrightarrow{\Gamma}^{A}_{|S}$ the total number of infections traveling through an oriented edge from $C$ to $D$ for the contact process $\xi^A_{|S}$ started from the configuration $A$ and restricted to $S$. Similarly, define
$\overleftarrow{\Gamma}^{A}_{|S}$ as the number of infections traveling through an oriented edge from $D$ to $C$ and
\begin{equation*}
\Gamma^A_{|S} \defeq \overrightarrow{\Gamma}^{A}_{|S} + \overleftarrow{\Gamma}^{A}_{|S}
\end{equation*}
the total number of infections travelling between $C$ and $D$. For any initial infected set $A \subset S$, we have
\begin{equation}\label{geomcrossing}
\P\left\{\Gamma^A_{|S} \geq k\right\} \leq \left(1 - \frac{1}{4^{r(C)}}\right)^k\quad\text{for all $k = 0,1,\ldots$}
\end{equation}
\end{step}
By stochastic monotonicity of the contact process stated in \eqref{couplingCP}, it suffices to prove the result for $A = S$. We consider a modification of $\xi^{S}_{|S}$ where the process resets every time an infection travels between $C$ and $D$, \emph{i.e.} each time an infection travels between $C$ and $D$, we start again from every site of $S$ infected. The number of resets  $\tilde{\Gamma}$ for this process stochastically dominates $\Gamma^S_{|S}$. Moreover, using the estimates from Step $2$ and Step $3$, after each reset there is probability at least $\frac{1}{4^{r(C)}}$ that the process dies out before any infection can travel between $C$ and $D$. Hence $\tilde{\Gamma}$ is dominated by a geometric random variable with parameter  $\frac{1}{4^{r(C)}}$ yielding \eqref{geomcrossing}.

\begin{step}[Main estimate \eqref{main-exit}] We have
\begin{equation*}
\E\left[
\begin{gathered}
\text{Total number of infections exiting through $\partial S$}\\
\text{for the contact process $\xi^C_{|S}$.}
\end{gathered}
\right]
\leq \frac{1}{2}e^{-\tilde{r}(C)^{1.01}}.
\end{equation*}
\end{step}
We consider the contact process $\xi^{C}_{|S}$. Let $\tau_{1},\tau_2,\ldots,\tau_{\overrightarrow{\Gamma}^{C}_{|S}}$ denote the times where an infection crosses from $C$ to $D$ and let $x_1,\ldots, x_{\overrightarrow{\Gamma}^{C}_{|S}}$ denote the sites of $D$ where these infections arrive. We can decompose the process $\xi^C_{|S}$ inside $D$ as a superposition of contact processes $\tilde{\xi}^{x_i}_{|D}$ \emph{i.e.}
\[
\xi^{C}_{|S}(t) \cap D = \bigcup_{
\substack{
1\leq i \leq \overrightarrow{\Gamma}^{C}_{|S}\\
\tau_i \leq t
}}
\tilde{\xi}_{|D}^{x_i}(t - \tau_i).
\]
Since these contact process are created with the same graphical construction, they are not independent however, conditionally on the event $\Big\{ \overrightarrow{\Gamma}^{C}_{|S} \geq k,\; x_i = x \Big\}$, the process $(\tilde{\xi}^{x_i}_{|D}(t - \tau_i), t\geq \tau_i)$ has the same law as a contact process restricted to $D$ and started with only the site $x_i$ infected. On the other hand, we can write
\begin{equation*}
\E\left[
\begin{gathered}
\text{Total number of infections exiting $\partial S$}\\
\text{for the contact process $\xi^C_{|S}$}
\end{gathered}
\right]
 \leq \E\left[ \sum_{k=1}^{\overrightarrow{\Gamma}^{C}_{|S}} \tilde{N}_k \right]
\end{equation*}
where $\tilde{N}_i$ is the total number of infections generated by $\tilde{\xi}^{x_i}_{|D}$ exiting $D$  through $\partial S$.
Thus, combining the results of Step $1$ and Step $4$, we get
\begin{multline*}
\E \left[
\begin{gathered}
\text{Total number of infections exiting $\partial S$}\\
\text{for the contact process $\xi^C_{|S}$}
\end{gathered}
\right]
\leq \sum_{k=1}^\infty \sum_{
 \substack{
 x\in D\\
 d(x,C)=1}
 } \E\left[ \ind_{\{ \overrightarrow{\Gamma}^{C}_{|S}\geq k,\; x_k = x \}} \tilde{N}_k \right]\\
\begin{aligned}
 & =
 \sum_{k=1}^\infty \sum_{
 \substack{
 x\in D\\
 d(x,C)=1}
 }
\P\left\{ \overrightarrow{\Gamma}^{C}_{|S}\geq k,\; x_k = x \right\}
\E\left[
\begin{gathered}
\text{Total number of infections exiting}\\
\text{$\partial S$ for the contact process $\xi^{x}_{|D}$}
\end{gathered}
\right]\\
&\leq
 \sum_{k=1}^\infty \sum_{
 \substack{
 x\in D\\
 d(x,C)=1}
 }
\P\left\{ \overrightarrow{\Gamma}^{C}_{|S}\geq k,\; x_k = x \right\} e^{-1.01 \tilde{r}(C)^{1.01}}\\
&=
 \sum_{k=1}^\infty
\P\left\{ \overrightarrow{\Gamma}^{C}_{|S}\geq k\right\} e^{-1.01 \tilde{r}(C)^{1.01}}\\
&\leq \sum_{k=1}^\infty  \left(1 - \frac{1}{4^{r(C)}}\right)^k e^{-1.01 \tilde{r}(C)^{1.01}}\\
& \leq 4^{r(C)} e^{-1.01 \tilde{r}(C)^{1.01}}\\
&\leq \frac{1}{2} e^{-\tilde{r}(C)^{1.01}}.
\end{aligned}
\end{multline*}

\begin{step}[Main estimate \eqref{main-time}] We have
\begin{equation*}
\E\left[
\text{Extinction time of the contact process $\xi^C_{|S}$ }
\right]
\leq e^{3\tilde{r}(C)}.
\end{equation*}
\end{step}

We use the same idea as in Step $5$. First, we write
\begin{align*}
\E\left[
\text{Extinction time of $\xi^C_{|S}$ }
\right]
& = \E\left[\int_{0}^\infty \ind_{\{\xi_{|S}^C(t) \neq \emptyset \} }dt\right]\\
&\leq \E\left[\int_{0}^\infty \ind_{\{\xi_{|S}^C(t)\cap C \neq \emptyset \} }dt\right] + \E\left[\int_{0}^\infty \ind_{\{\xi_{|S}^C(t)\cap D \neq \emptyset \} }dt\right].
\end{align*}
Recalling the notation of Step $5$, and denoting by $\tilde{T}_i$ the extinction time of $\tilde{\xi}^{x_i}_{|D}$  we have
\begin{equation*}
\E\left[\int_{0}^\infty \ind_{\{\xi_{|S}^C(t)\cap D \neq \emptyset \} }dt\right] \leq
\E\left[ \sum_{k=1}^{\overrightarrow{\Gamma}^{C}_{|S}} \tilde{T}_k \right].
\end{equation*}
Copying the arguments we used in the previous step, we find that
\begin{eqnarray}
\nonumber\E\left[\int_{0}^\infty \ind_{\{\xi_{|S}^C(t)\cap D \neq \emptyset \} }dt\right] &\leq&
\sum_{k=1}^{\infty} \sum_{
 \substack{
 x\in D\\
 d(x,C)=1}
 }
\!\!\!\!\P\left\{ \overrightarrow{\Gamma}^{C}_{|S}\geq k,\; x_k = x \right\}\! \E\left[\hbox{Extinction time of $\xi^{x}_{|D}$}\right]\\
\nonumber&\leq& 2 c \sum_{k=1}^{\infty}  \P\left\{ \overrightarrow{\Gamma}^{C}_{|S}\geq k \right\}\\
\label{last1}&\leq& 2 c 4^{r(C)}.
\end{eqnarray}
It remains to bound the expected total time during when the contact process has an infected site in $C$. Again, we decompose the process $\xi^C_{|S}$ inside $C$ as a superposition of contact processes. More rigorously, let $\gamma_{1},\gamma_2,\ldots,\gamma_{\overleftarrow{\Gamma}^{C}_{|S}}$ denote the times where an infection crosses from $D$ to $C$ and let $y_1,\ldots, y_{\overleftarrow{\Gamma}^{C}_{|S}}$ denote the sites in $C$ where these infections arrive. We write $\xi_{|S}^{C}\cap C$ as a superposition of contact processes $\hat{\xi}^{y_i}_{|C}$:
\[
\xi^{C}_{|S}(t) \cap C = \xi^{C}_{|C}(t) \cup \bigcup_{
\substack{
1\leq i \leq \overleftarrow{\Gamma}^{C}_{|S}\\
\gamma_i \leq t}
} \hat{\xi}_{|D}^{y_i}(t - \gamma_i).
\]
This yields
\begin{equation*}
\E\left[\int_{0}^\infty \ind_{\{\xi_{|S}^C(t)\cap C \neq \emptyset \} }dt\right] = \E[T] + \sum_{k=1}^{\infty}\E\left[ \ind_{\{
\overleftarrow{\Gamma}^{C}_{|S} \geq k \} }\hat{T}_k \right]
\end{equation*}
where $T$ is the extinction time of $\xi^C_{|C}$ and $\hat{T}_i$ is the extinction time of $\hat{\xi}^{y_i}_{|C}$. Conditionally on
the event $\left\{ \overleftarrow{\Gamma}^{C}_{|S} \geq k \right\}$, the process $\hat{\xi}^{y_k}_{|C}$ is stochastically dominated by $\xi_{|C}^C$. Thus, using the estimates of Step $3$ and Step $4$, we get
\begin{eqnarray}
\nonumber\E\left[\int_{0}^\infty \ind_{\{\xi_{|S}^C(t)\cap C \neq \emptyset \} }dt\right]
&\leq& \E[T] + \sum_{k=1}^{\infty}\P\left\{ \overleftarrow{\Gamma}^{C}_{|S} \geq k \right\}\E[T]\\
\label{last2}&\leq& 8^{r(C)}.
\end{eqnarray}
Finally, putting together \eqref{last1} and \eqref{last2}, we find that
$$
\E\left[
\hbox{Extinction time of $\xi^C_{|S}$ }
\right] \leq 2c4^{r(C)} + 8^{r(C)} \leq e^{3 r(C)}
$$
which completes Step $6$ and finishes the proof of the main estimates.
\end{proof}


\section{Questions and possible extensions. }
\label{sec:questions}
There are several natural questions left open in the paper -- some have already been stated. Here we present a few more which, in our opinion, might be interesting to look at.
\medskip

First, we have seen that the critical parameter $p_c$ for Bernoulli CMP on $d$-dimensional lattices is non trivial. Can this result be generalized to other graphs? We believe this to be true with minimal assumptions on the graphs and make the following conjecture:
\begin{conj}
Bernoulli CMP on any graph with bounded degrees has a non-trivial phase transition.  
\end{conj}   
A first step to prove this assertion might be to show the result for trees with bounded degrees. By coupling, it suffices to consider regular trees, which seems to be a much easier problem than the general case. However, it does not seem straightforward to extend the result to general graphs since (contrarily to the contact process for example) we cannot directly compare Bernoulli CMP on a given graph to Bernoulli CMP on a universal cover of the same graph.  
\medskip

A similar question applies for Continuum CMP: given a graph $G$, what are the conditions on the distribution of the radii $r$ to ensure the existence of a sub-critical phase? If $G$ has exponential growth, it is clear that $r$ should, at least, admit some exponential moments. Is this sufficient, at least for trees? 
\medskip

Concerning the model of degree-weighted CMP, an important setting is that of Galton-Watson trees for which we expect:
\begin{conj}\label{conjGW} Let $\alpha \geq 1$ and let $G$ be a Galton-Watson tree with reproduction law $B$ such that $\E[\exp(c B^{\alpha})] < \infty$ for any $c >0$. Then, degree weighted CMP on $G$ with expansion exponent $\alpha$ has a non-trivial phase transition. 
\end{conj}   
In view of Theorem \ref{TheoCMPtoCP}, this conjecture, if true, implies that the contact process on Galton-Watson trees has a non trivial phase transition whenever the progeny distribution $B$ has very light tails. This is a work in progress. 
\medskip

There are also many questions regarding finer percolation properties of the CMP which might be interesting to study. For example, is there percolation at criticality? When $p\neq p_c$, what is the tail distribution of the size of finite clusters? Is the decay faster than exponential?
\medskip

Another possible direction of investigation is to consider more general definitions of the CMP. As stated at the end of Section \ref{sec:CMP}, it is possible to generalize the notion of admissibility. For example, what happens if we consider an expansion exponent smaller than $1$. Do we have $p_c(\Z) = 1$ for Bernoulli CMP whenever $\alpha < 1$? 
\bigskip

Finally, concerning the connection between the CMP and the contact process, it would be extremely  satisfying to extend Theorem \ref{TheoCMPtoCP} to $\alpha = 1$. Assuming Conjecture \ref{conjGW}, this would imply the existence of a sub-critical phase for 
the contact process on any Galton-Watson tree whose progeny distribution admits exponential moments of all orders. This is in particular the case of the Poisson distribution appearing in the limit of Erd\H{o}s-Rényi random graphs. Conversely, if degree weighted CMP always has an infinite cluster for $\alpha =1$, does this imply that the critical infection rate for the contact process is zero?

\begin{ack}A.S. would like to thank J.-B. Gouéré for stimulating discussions concerning continuum percolation and related models. 
\end{ack}

\addcontentsline{toc}{section}{References}
\bibliographystyle{abbrv}
\bibliography{CMPandContact}

\end{document}